\documentclass[12pt]{article}
	\usepackage{amssymb,amsfonts,amsmath,amsthm, breqn, color, float, mathtools, url}
	\usepackage{caption}
	\newcommand{\E}{E}
	\newcommand{\Au}{Au}

	\newcommand{\pp}{{\mathbb P}}
	\newcommand{\qq}{{\mathbb Q}}

	\newcommand{\witi}{\widetilde}

	\newcommand{\nn}{{\mathbb N}}
	\newcommand{\ee}{{\mathbb E}}
	\newcommand{\hh}{{\mathbb H}}
	\newcommand{\rr}{{\mathbb R}}
	\newcommand{\cc}{{\mathbb C}}

	\newcommand{\cale}{{\mathcal E}}

	\newcommand{\calr}{{\mathcal R}}

	\newcommand{\veps}{\varepsilon}

	\newcommand{\beq}{\begin{eqnarray*}}
		\newcommand{\feq}{\end{eqnarray*}}
	\newcommand{\beqn}{\begin{eqnarray}}
	\newcommand{\feqn}{\end{eqnarray}}

	\newtheorem{theorem}{Theorem}
	\newtheorem*{conj*}{Conjecture}
	\makeatletter \@addtoreset{theorem}{section}\makeatother
	
	\newtheorem{definition}[theorem]{Definition}
	\newtheorem{lemma}[theorem]{Lemma}
	
	\newtheorem*{theorema*}{Theorem~A}
	\newtheorem*{theoremb*}{Theorem~B}
	\newtheorem*{theoremc*}{Theorem~C}
	\newtheorem*{theoremd*}{Theorem~D}
	\newtheorem*{theoreme*}{Theorem~E}
	\newtheorem*{theoremf*}{Theorem~F}
	\newtheorem*{cld*}{Condition $\mbox{LD}_d$}
	\newtheorem*{theorem*}{Theorem}
	
	\newtheorem{proposition}[theorem]{Proposition}
	\newtheorem{corollary}[theorem]{Corollary}
	
	\newtheorem{remark}[theorem]{Remark}
	\newtheorem{example}[theorem]{Example}

	\usepackage{algorithm}
	\usepackage{graphicx}
	\usepackage[noend]{algpseudocode}
	\makeatletter
	\def\BState{\State\hskip-\ALG@thistlm}
	\makeatother

	\DeclareCaptionLabelFormat{noname}{#2}
	\newlength\myindent
	\title{Finite automata, probabilistic method, and occurrence enumeration of a pattern in words and permutations}
	\author{Toufik~Mansour\thanks{ Department of Mathematics, University of Haifa, 199 Abba Khoushy Ave, 3498838 Haifa, Israel;
	\newline e-mail: tmansour@univ.haifa.ac.il} 
	\and 
	Reza~Rastegar\thanks{Occidental Petroleum Corporation, Houston, TX 77046 and Departments of Mathematics
	and Petroleum Engineering, University of Tulsa, OK 74104, USA - Adjunct Professor; e-mail:  reza\_rastegar2@oxy.com} 
	\and 
	Alexander~Roitershtein \thanks{Department of Statistics, Texas A\&M University, College Station, TX 77843, USA; 
	\newline e-mail: alexander@stat.tamu.edu}
	}
	
\begin{document}
	\maketitle
	\begin{abstract} 
	The main theme of this paper is the enumeration of the occurrence of a pattern in words and permutations. We mainly focus on asymptotic properties of the sequence $f_r^v(k,n),$ the number of $n$-array $k$-ary words that contain a given pattern $v$ exactly $r$ times. In addition, we study the asymptotic behavior of the random variable $X_n,$ the number of pattern occurrences in a random $n$-array word. The two topics are closely related through the identity $P(X_n=r) = $ $\frac{1}{k^n}f_r^v(k,n).$ In particular, we show that for any $r\geq 0,$ the Stanley-Wilf sequence $\bigl(f_r^v(k,n)\bigr)^{1/n}$ converges to a limit independent of $r,$ and determine the value of the limit. We then obtain several limit theorems for the distribution of $X_n,$ including a CLT, large deviation estimates, and the exact growth rate of the entropy of $X_n.$ Furthermore, we introduce a concept of weak avoidance and link it to a certain family of non-product measures on words that penalize pattern occurrences but do not forbid them entirely. We analyze this family of probability measures in a small parameter regime, where the distributions can be understood as a perturbation of a uniform measure. Finally, we extend some of our results for words, including the one regarding the equivalence of the limits of the Stanley-Wilf sequences, to pattern occurrences in permutations.
	\end{abstract}
	\section{Introduction and main results}
	Pattern occurrence enumeration is a central topic in modern combinatorics, see for instance the monographs \cite{bona, analcombin, HM, Kbook}. In this paper, we are primarily concerned with pattern occurrence problem for words, however, we provide the extension of certain results in the context of permutations. We define \emph{words} as finite arrays of letters from an alphabet $[k]:=\{1,\ldots,k\},$ for some given $k\in\nn.$ A pattern is any distinguished word, and occurrence of a pattern $v$ in a word $w$ is a subsequence of letters in $w$ (not necessarily consecutive) that are in the same relative order as the letters in $v.$ For instance, the word $w=37451554$ has four occurrences of the pattern $v=1332,$ namely $3**5*5*4,$ $3**5**54,$ $3****554,$ and $****1554.$ See Subsection~\ref{ns} for a more formal introduction of the concept. Occurrences of patterns in permutations are defined similarly, see the beginning of Section~\ref{astrid} for details.
	\par 
	Suppose that the alphabet $[k]$ and a pattern $v\in [k]^\ell$ are given, and that exactly $d\leq \ell$ distinct letters are used to form the pattern $v$. 
	For instance, if $k=7$ and $v=35731,$ then $\ell=5$ and $d=4.$  Our main object of interest is the frequency sequence $f_r^v(k,n),$ namely 
	the number of words in $[k]^n$ that contain the pattern $v$ exactly $r$ times. We also study the asymptotic behavior of the partial sums $g_r^v(k,n)=\sum_{j\leq r} f_r^v(k,n)$ and $X_n,$ the number of occurrences of $v$ in a random word distributed uniformly over $[k]^n$. Remark that the distribution of the random variable $X_n$ is related to the sequences $f_r^v(k,n)$ and $g_r^v(k,n)$ through the identities
	\beqn \label{fgP}
	P(X_n=r) =\frac{1}{k^n}f_r^v(k,n) \qquad \mbox{\rm and}  \qquad P(X_n\leq r) =\frac{1}{k^n}g_r^v(k,n).
	\feqn
	\par
	The starting point of our study is the celebrated Stanley-Wilf conjecture which states that the number of permutations 
	of size $n$ avoiding a pattern grows exponentially. The conjecture was settled by Marcus and Tardos \cite{MT} in 2004, 
	see \cite{wf1, fox, Kbook, wf-survey} for a review of the history and recent developments in the field. The analogue of this 
	result for the words is the convergence of the series $(f_0^v(k,n))^{1/n}$. This was proved by Br\"and\'{e}n and Mansour in \cite{BM} via a combinatorial analysis of certain finite automata that generate words avoiding a given pattern. In fact, it was shown in 
	\cite{BM} that $\lim_{n\to\infty}(f_0^v(k,n))^{1/n}=d-1,$ where $d$ is the number of distinct letters in the pattern $v.$  In Section~\ref{faust}, we generalize this result to all $r\geq 0.$ Specifically, we show the following (as stated in Theorems~\ref{thm:fnrlim_word} and~\ref{word}):
	\begin{theorema*}
	\item [(a)] For any integer $r\geq 0,$
	\beq
	\lim_{n\to \infty}\bigl(f_r^v(k, n)\bigr)^{\frac{1}{n}} = \lim_{n\to \infty}\bigl(g_r^v(k, n)\bigr)^{\frac{1}{n}}=d-1,
	\feq
	where $d$ is the number of distinct letters in the pattern $v.$
	\item [(b)] Assume that $d>1.$ Then  for any $r\geq 0,$ there exist a positive integer $M_r\in\nn$ and real constants $C_r\in (0,\infty)$ and $K_r\geq 0$ such that
	\beq
	\lim_{n\to\infty} \frac{g_r^v(k, n)}{n^{M_r}(d-1)^n}=C_r\qquad \mbox{\rm  and}\qquad \lim_{n\to\infty} \frac{f_r^v(k, n)}{n^{M_r}(d-1)^n}=K_r.
	\feq
	\end{theorema*}
	We remark that in various examples with $d>1,$ we are able to verify $K_r>0.$
	Nevertheless, we believe that it may be zero in some cases, see the discussion in Section~\ref{wftl}. \par
	We also give the following extension of this result for permutations. Let $\xi$ be a given permutation 
	pattern of size $k$ and $f_r^\xi(n)$ denote the number of permutations of size $n$ that contain $\xi$ exactly $r$ times, $r \geq 0.$ 
	We have (Theorem~\ref{thm:fnrlim} below): 
	\begin{theoremb*}
	For any $r\in\nn,$ $\lim_{n\to\infty}(f_r^\xi(n))^{\frac{1}{n}}$ exists and is equal to $\lim_{n\to\infty} (f_0^\xi(n))^{\frac{1}{n}}.$
	\end{theoremb*}
	In contrast to the obtained results in the context of words, we cannot describe the exact structure of Wilf-Stanley type limits as a function of the parameters $(k,\xi)$ in a general form.
	\par
	The next result turns out to be a direct implication of Theorem~A. It is stated below as Theorem~\ref{entropy}.
	\begin{theoremc*}
	If $d>1,$ then, $\lim_{n\to\infty} \frac{H_{k,v}(n)}{n} =\log\frac{k}{d-1},$ where $H_{k,v}(n)$ is the entropy of $X_n.$ 
	\end{theoremc*}
	Loosely speaking, for a given $n$, the entropy $H_{k,v}(n)$ measures the amount of uncertainty in the value of the random variable $X_n.$ Consequently, the entropy sequence $H_{k,v}(.)$ is subadditive, namely $H_{k,v}(n+m)\leq H_{k,v}(n)+ H_{k,v}(m)$ because of the dependence of pattern occurrences each of other. The convergence of $\frac{H_{k,v}(n)}{n}$ is thus ensured by Fekete's subadditivity lemma. Theorem~\ref{entropy} then gives the precise value of this limit for an arbitrary pattern $v.$ 
	\par
	In Sections~\ref{weaka} and~\ref{new} we  study the asymptotic behavior of the sequence $(X_n)_{n\in\nn}.$
	In Section~\ref{new} we obtain a central limit theorem and several related asymptotic results for the distribution of $X_n.$ 
	The following result is an analogue of the CLT for permutations obtained by B\'{o}na in \cite{bona}. The bulk of the proof is an estimation of the variance of $X_n$ referred to as $\text{VAR}(X_n).$  The latter, together with general theorems of \cite{becite} and \cite{LDPJ}, yields also a Berry-Esseen type bound for the rate of convergence and large deviation estimates stated, respectively, in Corollaries~\ref{bess} and~\ref{wsldp}. The following is the content of Theorem~\ref{wclt}.
	\begin{theoremd*}
	Let $\mu_n=E(X_n)$ and $\sigma_n=\sqrt{\text{VAR}(X_n)}.$ Then $\mu_n=\binom{n}{\ell}\binom{k}{d}\frac{1}{k^\ell},$
	$\sigma_n\sim\bigl(\frac{\mu_n}{\sqrt{n}}\bigr),$ and $\frac{X_n-\mu_n}{\sigma_n}$ converges in distribution, as $n\to\infty,$
	to a standard normal random variable.
	\end{theoremd*}
	For a pattern of length $\ell,$ there are $\binom{n}{\ell}$ places in a word $w\in [k]^n$ where the pattern might occur. Enumerate them in an arbitrary way, and let $X_{n,i}(w)$ be the indicator of the event that the pattern occurs at the $i$-th place in $w.$ Choose a parameter $x\in [0,1]$ and consider the following partition function penalizing the occurrences of $v:$ 
	\beq
	c^v_{k,n}(x) =\sum_{w\in[k]^n} \prod_{i=1}^{\binom{n}{\ell}} \bigl(1-xX_{n,i}(w)\bigr)=\sum_{w \in [k]^n} (1-x)^{occ_v(w)}=\sum_{r\geq 0} f_r^v(k,n)(1-x)^r.
	\feq  
	Using this partition function, one can construct a Boltzmann distribution on $[k]^n$ as follows:  
	\beq
	\qq^{v,x}_{k,n}(A)=\frac{1}{c^v_{k,n}(x)}\sum_{w \in A}(1-x)^{occ_v(w)},\qquad A\subset [k]^n.
	\feq
	The probability measure $\qq^{v,x}_{k,n}(\,\cdot\,)$ penalizes words $w$ with a non-zero $occ_v(w)$ with the factor $(1-x)^{occ_v(w)},$ 
	but unless $x=1$ it doesn't forbid them completely. We refer to a random word $w$ distributed according to $\qq^{v,x}_{k,n}$ as \emph{weakly avoiding} the pattern $v.$ The construction and the terminology are inspired by their analogue in the theory of self-avoiding walks, where a similar construction 
	is used to penalize self-intersection of the path of a random walk and introduce weakly self-avoiding walks \cite{saw}. Similar construction for permutations is outlined in Section~\ref{wa-gf}. In the case of permutations and the inversion pattern $21,$ the above probability measure is a Mallow's distribution. Mallow's permutations have been studied by many authors, see, for instance, recent work \cite{crane, mall, pitman} and references therein. 
	\par
	We remark that when $x=0,$ the above results for $X_n$ hold under $\qq^{v,x}_{k,n}$ as $\qq^{v,x}_{k,n}$ is the uniform distribution over $[k]^n$. One would then expect that for a sequence $(x_n)_{n\in\nn}$ decaying to zero sufficiently fast, similar limit theorems hold for $\qq^{v,x_n}_{k,n}$. Indeed, by using perturbation techniques we prove this the following (see Theorem~\ref{inv1}):
	\begin{theoreme*} 
	The following holds for any $t\in\rr$ and a sequence of positive reals $(\rho_n)_{n\in\nn} $ such that
	$\gamma:=\lim_{n\to\infty} \frac{n^\ell}{\rho_n}\in [0,+\infty):$
	\item [(a)] $\lim_{n\to\infty} \ee^{v,\frac{1}{\rho_n}}_{k,n}(e^{\frac{tX_n}{n^\ell}})=\exp\Bigl[\frac{t}{k^\ell\ell !}\binom{k}{d}\Bigr].$
	\item [(b)] $\lim_{n\to\infty} \frac{1}{\sqrt{n}}\log\ee^{v,\frac{1}{\rho_n}}_{k,n}(e^{\frac{tX_n\sqrt{n}}{n^\ell}})=J_{k,v}t,$
	where $J_{k,v}$ are strictly positive constants.
	\item [(c)] Let $\qq_n(r)=\qq^{v,\frac{1}{\rho_n}}_{k,n}(X_n=r)$ and $\hh_n=-\sum_{r\geq 0}\qq_n(r)\log \qq_n(r)$
	be the entropy of $X_n$ under the law $\qq^{v,\frac{1}{\rho_n}}_{k,n}.$ Then 
	$\lim_{n\to\infty} \frac{\hh_n}{n}=\log\frac{k}{d-1}+\frac{\gamma}{k^\ell\ell !}\binom{k}{d}.$
	\end{theoreme*}
	
	Note that in the context of permutations, somewhat similar perturbative regimes for Mallow's permutations were recently studied in \cite{nora, mall,starr}. \par
	
	Another interesting result closely related to Theorem~D (Theorem~\ref{wclt} below) is a limit theorem dealing with a Poisson approximation of $X_n$ in the case when $d=d_n$ is a rapidly increasing function of $n.$ The result is an analogue for random words of \cite[Theorem~3.1]{crane} for random permutations, it is stated below as Theorem~\ref{wpois}.
	\begin{theoremf*}
	Suppose that sequences of natural numbers $(k_n)_{n\in\nn},$ $(\ell_n)_{n\in\nn},$ and $(d_n)_{n\in\nn}$ satisfy the following condition:
	\begin{itemize}
	\item[] There exist constants $A>0$ and $\beta>\frac{2}{2+\delta}$ such that $\min\{k_n,\ell_n\}\geq d_n \geq An^\beta$ for all $n\in \nn,$ 
	where $\delta=\liminf_{n\to\infty}\frac{d_n}{\ell_n}.$ 
	\end{itemize}
	Consider an arbitrary sequence of patterns $v_n\in [k_n]^{\ell_n},$ $n\in\nn,$ with $d_n$ distinct letters used to form $v_n.$
	Let $X_n=occ_{v_n}(W_n),$ where $W_n$ is drawn at random from $[k_n]^n.$ Then
	\beq
	\lim_{n\to\infty}\Bigl|\frac{f_r^{v_n}(k_n,n)}{k_n^n}-\frac{\mu_n^r e^{-\mu_n}}{r!}\Bigr|=0,
	\feq
	for any integer $r\geq 0.$
	\end{theoremf*}	
	The paper is structured as follows. Section~\ref{limits} is devoted to pattern occurrences in words. 
	The framework is formally introduced in Section~\ref{ns}. In Section~\ref{faust} we study the sequences $f_r^v(k,n)$ and $g_r^v(k,n),$ $r\geq 0.$ The generating functions are explicitly computed for several examples using the automata approach and the transfer matrix method. The Stanley-Wilf limits of $f_r^v(k,n)$ and $g_r^v(k,n)$ are studied in Section~\ref{wftl}. Section~\ref{weaka} is devoted
	to the study of words weakly avoiding a pattern. Section~\ref{new} contains various limit theorems for the distribution of the random variable $X_n.$ Finally, within the framework of permutations the Stanley-Wilf type limits and words weakly avoiding a pattern are discussed in Section~\ref{astrid}.
	 	
	\section{Pattern occurrences in words}
	\label{limits}	
	In this section we focus on pattern occurrences in words and study the asymptotic behavior of $f_r^v(k, n)$ and $X_n.$ 
	The section is divided into five subsections. We begin with notation. Section's organization is discussed in more detail 
	at the end of Section~\ref{ns}.  
	\subsection{Notation and settings}
	\label{ns} 
	Let $\nn$ and $\nn_0$ denote, respectively, the set of natural numbers and the set of non-negative integers, that is $\nn_0=\nn \cup \{0\}.$ For a given set $A,$ $\#A$ is the cardinality of $A.$ For any given
	$k\in \nn,$ we denote the set $\{1,2, \cdots k\}$ by $[k]$  and refer to it as an \emph{alphabet} and to its elements as \emph{letters}. A \emph{word} of length $n,$ is an element of $[k]^n,$ $n\in\nn.$ A \emph{language} $[k]^{*} := \cup_{n=0}^{\infty} [k]^n$ is the set of all words compound of letters in an alphabet $[k].$ We adopt the convention that $[k]^0=\{\epsilon\},$ where $\epsilon$ is an empty word. For any $A\subset \nn_0$ we denote by $[k]^A$ the union $\cup_{j\in A} [k]^j.$ For instance, $[k]^{\geq n}=\cup_{j\geq n} [k]^j$ and $[k]^{\leq n}=\cup_{j\leq n} [k]^j.$ We write a word $w\in [k]^n$ in the form $w=w(1)\cdots w(n),$ where $w(i)$ is the $i$-the letter of $w.$ The \emph{concatenation} of two words $w\in [k]^n$ and $v\in [k]^m$ is the word
	$wv:=w(1)\cdots w(n)v(1)\cdots v(m).$ For instance, the concatenation of $w=20$ and $v=19$ is $wv=2019.$  A \emph{pattern} is any distinguished word in the underlying language $[k]^*.$
	\par
	Let us now fix integers $k>0,$ $\ell\geq 2,$ and a pattern $v$ in $[k]^\ell.$ These parameters are considered to be given and fixed throughout the rest of Section~\ref{limits}. An important characteristic of the pattern turns out to be the number of distinct letters used to compound it. We will denote this number by $d.$ For instance, if $v=33415,$ then $\ell=5$ and $d=4.$
	\par
	For a word $w\in [k]^n$ with $n\geq \ell,$ an occurrence of the pattern $v$ in $w$ is a sequence of $\ell$ indices
	$1\leq j_1<j_2<\dots<j_\ell\leq n $ such that the \emph{subword} $w(j_1)\cdots w(j_\ell)\in [k]^\ell$
	is order-isomorphic to the word $v,$ that is
	\beq
	w(j_p)<w(j_q)\Longleftrightarrow  v_p<v_q\qquad \forall\,1\leq p,q\leq \ell
	\feq
	and
	\beq
	w(j_p)=w(j_q)\Longleftrightarrow  v_p=v_q\qquad \forall\,1\leq p,q\leq \ell.
	\feq
	For a word $w \in [k]^*,$ we denote by $occ_v(w)$ the number of occurrences of $v$ in $w.$ For instance, if $v$ is the \emph{inversion} $21$
	and $w=35239,$ then $occ_v(w)=3$ (for the following three occurrences of pairs of letters which appear in the reverse order: $w(1)w(3)=32,$ $w(2)w(3)=52,$ and $w(2)w(4)=53$). We say that a word $w\in [k]^*$ \emph{contains the pattern} $v$ exactly $r$ times, $r\in\nn_0,$ if $occ_v(w)=r.$  For $r\in \nn_0,$ we denote by $f_r^v(k,n)$ and $g_r^v(k,n),$ the number of words in $[k]^n$ that contain $v,$ respectively, exactly $r$ times and at most $r$ times. That is,
	\beqn
	\label{gfr}
	f_r^v(k,n)=\#\{w\in[k]^n:occ_v(w)=r\} \quad \mbox{\rm and}\quad  g_r^v(k,n)=\sum_{j=0}^r f_j^v(k,n).
	\feqn
	We define their corresponding generating functions as
	\beqn
	\label{GF}
	F_{k,n}^v(x)=\sum_{r \geq 0} f_r^v(k,n)x^r \quad \text{and}\quad G_{k,n}^v(x)=\sum_{r \geq 0} g_r^v(k,n)x^r.
	\feqn
	We remark that given $f_r^v(k,n)=0$ for $r>\binom{n}{\ell},$ $F_{k,n}^v(x)$ is a polynomial in $x$.
	Throughout this paper, $a_n\sim b_n,$ $a_n=O(b_n),$ and $a_n=o(b_n)$ for sequences $a_n$ and $b_n$ with elements that might depend on $k,r,\ell,d,v$ and
	other parameters, means that, respectively, $\lim_{n\to\infty}\frac{a_n}{b_n}=1,$ $\limsup_{n\to\infty}\bigl|\frac{a_n}{b_n}\bigr|<\infty,$
	and $\lim_{n\to\infty}\frac{a_n}{b_n}=0$ for all feasible values of the parameters when the latter are fixed. As usual, $a_n=\Theta(b_n)$ indicates 
	that both $a_n=O(b_n)$ and $b_n=O(a_n)$ hold true.   
	\par	
	The remainder of this section is divided into four subsections. In Section~\ref{faust} we study a finite state automaton that generates words 
	$w\in [k]^n$ with a given value of $occ_v(w).$	The words are then counted trough an application of the transfer-matrix method, 
	allowing us to evaluate $g_r^v(k,n)$ and subsequently $f_r^v(k,n)$ in several interesting cases. The results of Section~\ref{faust} 
	are then used in Section~\ref{wftl} to show that (see Theorem~\ref{thm:fnrlim_word}) for any $r\geq 0,$
	\beq
	\lim_{n\to \infty}\bigl(f_r^v(k, n)\bigr)^{\frac{1}{n}} = \lim_{n\to \infty}\bigl(g_r^v(k, n)\bigr)^{\frac{1}{n}}=d-1,
	\feq
	where $d$ is the number of distinct letters in the pattern $v.$  Theorem~\ref{thm:fnrlim_word} is the main result of this paper.  
	Remark that a similar result for permutations is given by Theorem~\ref{thm:fnrlim} in Section~\ref{ldpsw}. We refer to $\lim_{n\to \infty}\bigl(f_r^v(k, n)\bigr)^{\frac{1}{n}}$ and their counterparts for permutations in Theorem~\ref{thm:fnrlim} as Stanley-Wilf type limits. 
	\par 
	Finally, Sections~\ref{weaka} and ~\ref{new} deal with random words.
	Let $W_n$ be a permutation chosen at random from $[k]^n$ and $X_n=occ_v(W_n).$ In Section~\ref{new} we obtain a central limit theorem 
	and several related asymptotic results for the distribution of $X_n.$ The study of $X_n$ is, in principle, equivalent to the study of the sequences  
	$f_r^v(k,n)$ and $g_r^v(k,n)$ in view of the identities \eqref{fgP}. In Section~\ref{weaka} we introduce a notion of weak avoidance for an arbitrary word pattern. In Theorem~\ref{inv1} we obtain limit theorems for random words avoiding a pattern weakly. The distribution of $W_n$ is not uniform in this case, and we use  the CLT for the uniform case and perturbation techniques to derive the results.
	\subsection{Finite automata and pattern occurrences} 
	\label{faust}
	Given an integer $r\geq0$, we define an equivalence relation $\sim_{v;r}$  on $[k]^*$ as follows.
	We say that two words $w'$ and $w$ in $[k]^*$ are equivalent and write $w' \sim_{v;r} w$ if the following condition holds for all $u \in [k]^*:$
	\beqn
	\label{adef}
	occ_{v}(w'u)=m \quad \mbox{\rm if and only if}\quad  occ_{v}(wu)=m,\qquad \forall~m\leq r.
		\feqn
	For instance, if $k=2$, $r=1$ and $v=12$, then $1\not\sim_{v;r}11$ because $occ_{12}(12)=1$ and $occ_{12}(112)=2$. On the other hand, $11\sim_{v;r}111$ because $occ_{12}(11u)=occ_{12}(111u)=m$ for any $m=0,1$, and $u\in[2]^*$. We denote the equivalence class of a word $w$ by $\langle w \rangle_{v;r}$. For simplicity in notation, we drop the indexes when context is clear. We remark that:
	\begin{itemize}
	\item[-] $w$ and $w'$ do not need to have the same length in order to be equivalent;
	\item[-] if $occ_v(w)>r$ and $occ_v(w')>r,$ then $w\sim_{v;r} w'.$
	\end{itemize}
	The latter observation implies that there is a unique equivalence class $\calr(v,r,k)$ such that
	$$\{w\in[k]^*:occ_v(w)>r\}\subset \calr(v,r,k).$$ Since the empty word $\epsilon$ is an element of the language $[k]^*,$ it
	follows from \eqref{adef} that if $occ_v(w)\leq r$ then
	\beq
	\langle w \rangle_{v;r} \subset \{w'\in[k]^*:occ_v(w')= occ_v(w)\}.
	\feq
	In particular,
	\beqn
	\label{calr}
	\calr(v,r,k)=\{w\in[k]^*:occ_v(w)>r\}.
	\feqn
	The following lemma shows that the equivalence of any two words can be checked with a finite number of steps.
	\begin{lemma}
	\label{blem1}
	Let $w'$ and $w$ be two words in $[k]^*.$ Then $w' \sim_{v;r} w$ if and only if
	\eqref{adef} holds for all $u \in [k]^{\leq r\ell}.$
	\end{lemma}
	\begin{proof}
	Let $\sim_{v;r}'$ be an equivalence relation on $[k]^*$ such that
	$w' \sim_{v;r}' w$ if and only if \eqref{adef} holds for all $u \in [k]^{\leq r\ell}.$
	Clearly, $w' \sim_{v;r} w$ implies $w' \sim_{v;r}' w$. On the other hand, if
	$w' \nsim_{v;r} w$ then there exists $u \in [k]^*$ such that $occ_{v}(w'u)=m_1$ and $occ_{v}(wu)=m_2$
	with $m_1\neq m_2$ and $m_1,m_2\leq r.$ Without loss of generality we may assume that $m_1<m_2\leq r$.
	The occurrences of $v$ in $wu$ can use at most $m_2\ell$ letters of $u.$ Thus there is a subsequence
	$u'$ of $u$ of length at most $m_2\ell$ such that $occ_v(w'u')\leq m_1$ and $occ_v(wu')=m_2$, and hence $w' \nsim_{v;r}' w$.
	\end{proof}
	Let $\cale(v,r,k)$ be the set of all equivalence classes of $\sim_{v;r}.$ Note that by Lemma \ref{blem1} the number of equivalence classes is finite.
	Recall $\calr(v,r,k)$ from \eqref{calr}, and let
	\beq
	E(v,r,k)=\cale(v,r,k)\backslash\{\calr(v,r,k)\}
	\feq
	denote the set of equivalence classes excluding $\calr(v,r,k).$  By the definition,
	\beq
	E(v,r,k)=\bigcup_{\{w\in [k]^*:occ_v(w)\leq r\}}\, \langle w \rangle_{v;r}.
	\feq
	We next introduce the key tool in our proofs in this section. 	
	\begin{definition}
	\label{defauto}
	Given an integer $r\geq 0,$ we denote by $\Au(v,r,k)$ a \emph{finite automaton} \cite{Hop} such that
	\begin{itemize}
	\item The set of states of the automaton is $\E(v,r,k);$
	\item The input alphabet is $[k];$
	\item Transition function $\delta : \E(v,r,k) \times [k] \rightarrow \E(v,r,k)$ is given by the rule
	$\delta(\langle w \rangle,a)=\langle wa \rangle;$
	\item The {\em initial state} is $\langle\epsilon\rangle,$  where $\epsilon$ denotes the empty word;
	\item All states are final states.
	\end{itemize}
	\end{definition}
	We identify the automaton $A(v,r,k)$ with a (labeled) directed graph
	with vertices in $\E(v,r,k)$ such that there is a labeled edge
	$\stackrel{a}{\longrightarrow}$ from $\langle w \rangle$ to
	$\langle w' \rangle$ if and only if $wa \sim_{v,r} w'$.
	\begin{example}
	\label{ae}
	Consider the case $v=123$, $k=3,$ and $r=1.$ The set of equivalence classes $\E(123,1,3)$ is given by
	\beq
	\E(123,1,3)=\{\langle\epsilon\rangle,\langle1\rangle,\langle11\rangle,\langle12\rangle,
	\langle112\rangle,\langle123\rangle\}.
	\feq
	The labeled graph associated with the automaton $\Au(123,1,3)$ is
	
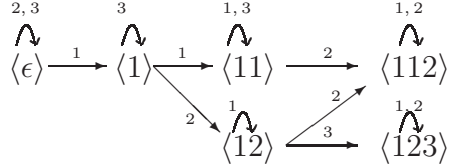
\begin{center}
	\begin{picture}(150,80)
	\put(0,-50){
	\put(0,90){$\langle\epsilon\rangle$}
	\put(40,90){$\langle1\rangle$}
	\put(80,90){$\langle11\rangle$}
	\put(140,90){$\langle112\rangle$}
	\put(80,60){$\langle12\rangle$}
	\put(140,60){$\langle123\rangle$}
	\put(15,94){\vector(1,0){22}\put(-13,3){\tiny$1$}}\put(55,94){\vector(1,0){22}\put(-13,3){\tiny$1$}}
	\put(105,94){\vector(1,0){27}\put(-13,3){\tiny$2$}}\put(55,94){\vector(1,-1){25}\put(-13,-22){\tiny$2$}}
	\put(105,64){\vector(1,0){27}\put(-13,3){\tiny$3$}}
	\put(105,64){\vector(4,3){30}\put(-13,17){\tiny$2$}}
	\put(3,102){\qbezier(0,0)(4,15)(7,0)
	\multiput(6.4,0)(.1,.2){12}{\put(0,0){\tiny.}}
	\multiput(6.4,0)(-.2,.1){12}{\put(0,0){\tiny.}}
	\put(-2,12){\tiny$2,3$}}
	\put(43,102){\qbezier(0,0)(4,15)(7,0)
	\multiput(6.4,0)(.1,.2){12}{\put(0,0){\tiny.}}
	\multiput(6.4,0)(-.2,.1){12}{\put(0,0){\tiny.}}
	\put(-2,12){\tiny$3$}}
	\put(83,102){\qbezier(0,0)(4,15)(7,0)
	\multiput(6.4,0)(.1,.2){12}{\put(0,0){\tiny.}}
	\multiput(6.4,0)(-.2,.1){12}{\put(0,0){\tiny.}}
	\put(-2,12){\tiny$1,3$}}
	\put(148,102){\qbezier(0,0)(4,15)(7,0)
	\multiput(6.4,0)(.1,.2){12}{\put(0,0){\tiny.}}
	\multiput(6.4,0)(-.2,.1){12}{\put(0,0){\tiny.}}
	\put(-2,12){\tiny$1,2$}}
	\put(148,69){\qbezier(0,0)(4,15)(7,0)
	\multiput(6.4,0)(.1,.2){12}{\put(0,0){\tiny.}}
	\multiput(6.4,0)(-.2,.1){12}{\put(0,0){\tiny.}}
	\put(-2,8){\tiny$1,2$}}
	\put(85,69){\qbezier(0,0)(4,15)(7,0)
	\multiput(6.4,0)(.1,.2){12}{\put(0,0){\tiny.}}
	\multiput(6.4,0)(-.2,.1){12}{\put(0,0){\tiny.}}
	\put(-2,8){\tiny$1$}}}
	\end{picture}
	\end{center}
	\end{example}
	$\mbox{}$
	\par
	The automata serves for us as a bridge between the formal language theory and theory of computing on one side and the asymptotic theory of algebraic functions on the other. See, for instance, \cite{autc, analcombin} and references therein for background.
	\par
	We exploit the link between asymptotic properties of rational functions and the structure of associated regular languages to study the generating functions $F_{r,k}^v(x)$ and $G_{r,k}^v(x)$
	of the sequences $f_r^v(k,n)$ and $g_r^v(k,n)$ defined in \eqref{GF}, and subsequently the asymptotic behavior of these sequences, as $n$ tends to infinity. The class of automata $Au(v,0,k)$ has been introduced in \cite{BM}. Our results in this subsection (Lemmas~\ref{lemloops} and \ref{gf} below) are extensions of the corresponding results in Section~2 of \cite{BM}.
	\par
	It is straightforward to verify (cf. \cite[p.~256]{HM}) that one can order the states of the automaton $Au(v,r,k)$ as $s_1,x_2, \ldots, s_p,$
	$p=\#\E(v,r,k),$ so that if $i<j$ then there is no path from the state $s_j$ to the state $s_i$. {\em Transition matrix} $T(v,r,k)$ of
	$\Au(v,r,k)$ is the $p\times p$  matrix with non-negative integer entries defined by
	\beq
	[T(v,r,k)]_{ij}= \#\{a \in [k] : \delta(s_i,a)=s_j\}.
	\feq
	Thus $[T(v,r,k)]_{ij}$ counts the number of edges between $s_i$ and $s_j$, and $T(v,r,k)$ is triangular.
	The following observation reduces the study of the sequence $g_r^v(k,n),$ $n\in\nn,$ to the analysis of the matrix
	$T(v,r,k):$
	\beqn
	\label{a}
	\nonumber
	g_r^v(k,n)&=&\#\{\mbox{paths of length $n$ starting at $\langle \epsilon \rangle$ in the graph associated with $Au(v,r,k)$}\}
	\\
	&=&\sum_{j=1}^p [T^n]_{1j},
	\feqn
	where $T=T(v,r,k)$ and $p=\#E(v,r,k).$
		
	\begin{example}
	\label{ae1}
	Consider again the setup of Example~\ref{ae}, namely $v=123$, $k=3,$ and $r=1.$
	The transition matrix $T(123,1,3)$ is given by
	\beq
	\left(\begin{array}{llllll}
	2&1&0&0&0&0\\
	0&1&1&1&0&0\\
	0&0&2&0&1&0\\
	0&0&0&1&1&1\\
	0&0&0&0&2&0\\
	0&0&0&0&0&2\\
	\end{array}\right).
	\feq
	Thus the generating function for the number of $3$-ary words of length $n$ that contains $123$ at most once is given by
	\beqn
	\nonumber
	G_{1,3}^{123}(x)&=&\sum_{n\geq 0}g_1^{123}(3,n)x^n=e_1^t\sum_{n\geq 0}T(123,1,3)^nx^n(e_1+\cdots+e_6)
	\\
	\label{g3}
	&=&\frac{(x^4-8x^3+10x^2-5x+1)}{(1-2x)^3(1-x)^2},
	\feqn
	 where $e_i$ is the $i$-th standard unit vector (all coordinates are zero, except that the $i$-th
	 coordinate is one). Note that the  generating function for the number of $3$-ary words of length $n$ that avoids $123$ is given
	by $F_{0,3}^{123}(x)=\sum_{n\geq 0}f_0^{123}(3,n)x^n=\frac{3x^2-3x+1}{(1-2x)^3}$ (see \cite{B}.) Therefore, by virtue of \eqref{g3},
	\beq
	F_{1,3}^{123}(x)=\sum_{n\geq 0}f_1^{123}(3,n)x^n=\frac{x^3}{(1-2x)^2(1-x)^2}.
	\feq
	Applying arguments similar to the one we used in order to get \eqref{g3}, we find that
	\beq
	G_{1,4}^{123}(x)=\sum_{n\geq 0}g_1^{123}(4,n)x^n=\frac{(1-7x+22x^2-32x^3+16x^4-2x^5)}{(1-x)(1-2x)^5},
	\feq
	and
	\beq
	G_{1,5}^{123}(x)=\sum_{n\geq 0}g_1^{123}(5,n)x^n=\frac{(1-10x+48x^2-124x^3+170x^4-103x^5-3x^6+23x^7)}{(1-x)(1-2x)^7}.
	\feq
	\end{example}
	
	\begin{example}
	\label{k1}
	The equivalence classes of $\Au(12\cdots k,0,k)$ are given by $\langle\epsilon\rangle$ and
	$\langle12\cdots j\rangle$, where $j=1,2,\ldots,k-1$. The $\Au(12\cdots k,0,k)$ can be graphically represented as follows:
	
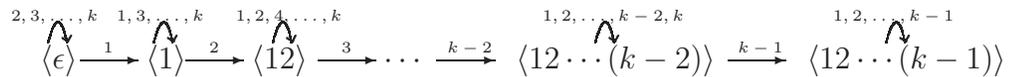
\begin{center}
	\begin{picture}(300,40)
	\put(0,12){ \put(0,0){$\langle\epsilon\rangle$}
	\put(40,0){$\langle1\rangle$} \put(80,0){$\langle12\rangle$}
	\put(180,0){$\langle12\cdots(k-2)\rangle$}
	\put(290,0){$\langle12\cdots(k-1)\rangle$}
	\put(15,3){\vector(1,0){22}\put(-13,3){\tiny$1$}}
	\put(55,3){\vector(1,0){22}\put(-13,3){\tiny$2$}}
	\put(105,3){\vector(1,0){22}\put(-13,3){\tiny$3$}}
	\put(130,0){$\cdots$}
	\put(150,3){\vector(1,0){22}\put(-18,3){\tiny$k-2$}}
	\put(260,3){\vector(1,0){22}\put(-18,3){\tiny$k-1$}}
	\put(3,10){\qbezier(0,0)(4,15)(7,0)
	\multiput(6.4,0)(.1,.2){12}{\put(0,0){\tiny.}}
	\multiput(6.4,0)(-.2,.1){12}{\put(0,0){\tiny.}}
	\put(-14,8){\tiny$2,3,\ldots,k$}}
	\put(43,10){\qbezier(0,0)(4,15)(7,0)
	\multiput(6.4,0)(.1,.2){12}{\put(0,0){\tiny.}}
	\multiput(6.4,0)(-.2,.1){12}{\put(0,0){\tiny.}}
	\put(-14,8){\tiny$1,3,\ldots,k$}}
	\put(88,10){\qbezier(0,0)(4,15)(7,0)
	\multiput(6.4,0)(.1,.2){12}{\put(0,0){\tiny.}}
	\multiput(6.4,0)(-.2,.1){12}{\put(0,0){\tiny.}}
	\put(-14,8){\tiny$1,2,4,\ldots,k$}}
	\put(210,10){\qbezier(0,0)(4,15)(7,0)
	\multiput(6.4,0)(.1,.2){12}{\put(0,0){\tiny.}}
	\multiput(6.4,0)(-.2,.1){12}{\put(0,0){\tiny.}}
	\put(-20,8){\tiny$1,2,\ldots,k-2,k$}}
	\put(320,10){\qbezier(0,0)(4,15)(7,0)
	\multiput(6.4,0)(.1,.2){12}{\put(0,0){\tiny.}}
	\multiput(6.4,0)(-.2,.1){12}{\put(0,0){\tiny.}}
	\put(-20,8){\tiny$1,2,\ldots,k-1$}}}
	\end{picture}
	\end{center}
	Therefore, $T(12\cdots k,0,k)$ is given by the matrix $(a_{ij})_{1\leq i,j\leq k}$ with $a_{ii}=k-1$ and
	$a_{i(i+1)}=1$ for all $i=1,2,\ldots,k,$ and the remaining entries equal to zero. Consequently,
	\beq
	\sum_{n\geq 0} f_0^{12\cdots k}(k,n)x^n=\sum_{j=0}^{k-1}\frac{x^j}{(1-(k-1)x)^{j+1}}.
	\feq
	\end{example}
		
	\begin{example}
	It is not hard to see that the equivalence classes of $\Au(12\cdots k,1,k)$ are given by $\langle\epsilon\rangle$,
	$\langle12\cdots j\rangle$ for $j=1,2,\ldots,k,$ and $\langle112\cdots j\rangle$ for $j=1,2,\ldots,k-1$. The automaton
	$\Au(12\cdots k,1,k)$ can be graphically represented as follows:
	
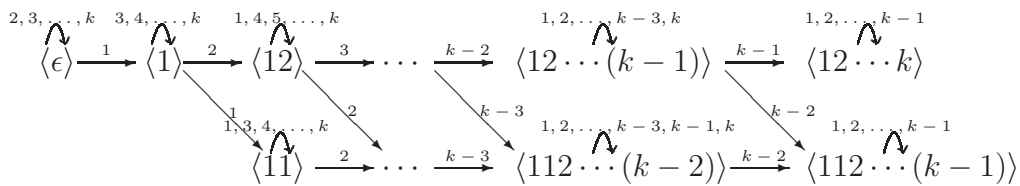
\begin{center}
	\begin{picture}(350,70)
	\put(0,50){ \put(0,0){$\langle\epsilon\rangle$}
	\put(40,0){$\langle1\rangle$}
	\put(80,0){$\langle12\rangle$}\put(80,-40){$\langle11\rangle$}
	\put(180,0){$\langle12\cdots(k-1)\rangle$}
	\put(180,-40){$\langle112\cdots(k-2)\rangle$}
	\put(290,0){$\langle12\cdots k\rangle$}
	\put(290,-40){$\langle112\cdots(k-1)\rangle$}
	\put(15,3){\vector(1,0){22}\put(-13,3){\tiny$1$}}
	\put(55,3){\vector(1,0){22}\put(-13,3){\tiny$2$}}
	\put(55,0){\vector(1,-1){30}\put(-13,-18){\tiny$1$}}
	\put(105,3){\vector(1,0){22}\put(-13,3){\tiny$3$}}
	\put(105,-37){\vector(1,0){22}\put(-13,3){\tiny$2$}}
	\put(100,0){\vector(1,-1){30}\put(-13,-17){\tiny$2$}}
	\put(150,0){\vector(1,-1){30}\put(-13,-17){\tiny$k-3$}}
	\put(260,0){\vector(1,-1){30}\put(-13,-17){\tiny$k-2$}}
	\put(130,0){$\cdots$}\put(130,-40){$\cdots$}
	\put(150,3){\vector(1,0){22}\put(-18,3){\tiny$k-2$}}
	\put(150,-37){\vector(1,0){22}\put(-18,3){\tiny$k-3$}}
	\put(260,3){\vector(1,0){22}\put(-18,3){\tiny$k-1$}}
	\put(262,-37){\vector(1,0){22}\put(-18,3){\tiny$k-2$}}
	\put(3,10){\qbezier(0,0)(4,15)(7,0)
	\multiput(6.4,0)(.1,.2){12}{\put(0,0){\tiny.}}
	\multiput(6.4,0)(-.2,.1){12}{\put(0,0){\tiny.}}
	\put(-14,8){\tiny$2,3,\ldots,k$}}
	\put(43,10){\qbezier(0,0)(4,15)(7,0)
	\multiput(6.4,0)(.1,.2){12}{\put(0,0){\tiny.}}
	\multiput(6.4,0)(-.2,.1){12}{\put(0,0){\tiny.}}
	\put(-14,8){\tiny$3,4,\ldots,k$}}
	\put(88,-31){\qbezier(0,0)(4,15)(7,0)
	\multiput(6.4,0)(.1,.2){12}{\put(0,0){\tiny.}}
	\multiput(6.4,0)(-.2,.1){12}{\put(0,0){\tiny.}}
	\put(-18,8){\tiny$1,3,4,\ldots,k$}}
	\put(88,10){\qbezier(0,0)(4,15)(7,0)
	\multiput(6.4,0)(.1,.2){12}{\put(0,0){\tiny.}}
	\multiput(6.4,0)(-.2,.1){12}{\put(0,0){\tiny.}}
	\put(-14,8){\tiny$1,4,5,\ldots,k$}}
	\put(210,10){\qbezier(0,0)(4,15)(7,0)
	\multiput(6.4,0)(.1,.2){12}{\put(0,0){\tiny.}}
	\multiput(6.4,0)(-.2,.1){12}{\put(0,0){\tiny.}}
	\put(-20,8){\tiny$1,2,\ldots,k-3,k$}}
	\put(210,-31){\qbezier(0,0)(4,15)(7,0)
	\multiput(6.4,0)(.1,.2){12}{\put(0,0){\tiny.}}
	\multiput(6.4,0)(-.2,.1){12}{\put(0,0){\tiny.}}
	\put(-20,8){\tiny$1,2,\ldots,k-3,k-1,k$}}
	\put(310,10){\qbezier(0,0)(4,15)(7,0)
	\multiput(6.4,0)(.1,.2){12}{\put(0,0){\tiny.}}
	\multiput(6.4,0)(-.2,.1){12}{\put(0,0){\tiny.}}
	\put(-20,8){\tiny$1,2,\ldots,k-1$}}
	\put(320,-31){\qbezier(0,0)(4,15)(7,0)
	\multiput(6.4,0)(.1,.2){12}{\put(0,0){\tiny.}}
	\multiput(6.4,0)(-.2,.1){12}{\put(0,0){\tiny.}}
	\put(-20,8){\tiny$1,2,\ldots,k-1$}}}
	\end{picture}
	\end{center}
	Hence $T(12\cdots k,1,k)$ is given by the matrix $(b_{ij})_{1\leq i,j\leq 2k}$
	with $a_{11}=k-1$, $a_{12}=1$, $a_{22}=k-2$, $a_{23}=a_{2,4}=1$, $a_{2i,2i}=k-1,$ $a_{2i,2i+2}=1,$
	$a_{2i-1,2i-1}=k-2,$ and $a_{2i-1,2i+1}=a_{2i-1,2i+2}=1$ for all $i=2,3,\ldots,k-1$, $a_{2k,2k}=a_{2k-1,2k-1}=k-1$,
	and the remaining entries equal to zero. Let $C=I-xT(12\cdots k,1,k)$. In view of \eqref{a}, we are interested in computing $e_1^tC^{-1}(e_1+\cdots+e_{2k})$. First, we solve the system $C{\mathfrak z} =e_1+\cdots+e_{2k}$,
	where ${\mathfrak z}={\mathfrak z}(x)$ is the vector ${\mathfrak z}=({\mathfrak z}_1,\ldots,{\mathfrak z}_{2k})^t$.
	By induction,
	\beq
	{\mathfrak z}_{2k-2j}(x)=\sum_{i=1}^{j+1}\frac{x^{i-1}}{(1-(k-1)x)^i}
	\feq
	and
	\beq
	{\mathfrak z}_{2k-1-2j}(x)&=&\frac{x^{j+1}}{(1-(k-1)x)(1-(k-2)x)^{j+1}}
	\\
	&&
	\qquad
	+\sum_{i=1}^{j+1}\frac{x^{i-1}}{(1-(k-2)x)^i}\left(1+x\sum_{s=1}^{j+1-i}\frac{x^{s-1}}{(1-(k-1)x)^s}\right),
	\feq
	for $j=1,2,\ldots,k-1$. Hence,
	\begin{align*}
	&\sum_{n\geq 0} g_1^{12\cdots k}(k,n)x^n=e_1^tC^{-1}(e_1+\cdots+e_{2k})={\mathfrak z}_1(x)\\
	&=\frac{1}{1-(k-1)x}+\frac{x}{(1-(k-1)x)(1-(k-2)x)}
	+\frac{x^2}{(1-(k-1)x)(1-(k-2)x)}({\mathfrak z}_3+{\mathfrak z}_4)\\
	&=\frac{1}{1-(k-1)x}+\frac{x}{(1-(k-1)x)(1-(k-2)x)}\\
	&+\frac{x^2}{(1-(k-1)x)(1-(k-2)x)}\sum_{i=1}^{k-1}\frac{x^{i-1}}{(1-(k-1)x)^i}+\frac{x^{k+1}}{(1-(k-1)x)^2(1-(k-2)x)^{k}}\\
	&+\frac{1}{1-(k-1)x}\sum_{i=1}^{k-1}\frac{x^{i+1}}{(1-(k-2)x)^{i+1}}\left(1+x\sum_{s=1}^{k-1-i}\frac{x^{s-1}}{(1-(k-1)x)^s}\right).
	\end{align*}
	Taking in account the result in Example~\ref{k1}, we conclude that the generating
	function for the number of $k$-ary words of length $n$ that contains $12\cdots k$ exactly once is given by
	\begin{align*}
	& F_{1,k}^{12\cdots k}(x)=\sum_{n\geq 0} f_1^{12\cdots k}(k,n)x^n=\frac{x^k}{(1-(k-1)x)^2(1-(k-2)x)^{k-1}} \\
	&\qquad +\frac{1}{(1-kx)(1-(k-2)x)}\left(1-(k-2)x-\frac{x^k}{(1-(k-1)x)^{k-1}}-\frac{x^6}{(1-(k-1)x)^5}\right).
	\end{align*}
	Note that $\lim_{x\rightarrow1/k}F_{1,k}^{12\cdots k}(x)=
	\frac{k((k+5)2^k+4)}{2^{k+1}}  $.
	Hence the minimal by absolute value pole of $F_{1,k}^{12\cdots k}(x)$ is $x=1/(k-1),$
	and it is of order $k-2$ when $k\geq6$. Thus (see, for instance, \cite{analcombin} or \cite{autc}), as $n\to\infty,$
	\beq
	f_1^{12\cdots k}(k,n)\sim \frac{n^{k-2}(k-1)^{n+3-k}}{(k-1)!} \quad \text{for} \quad k\geq 7.
	\feq
	For $k \leq 6$ we have:
	\beq
	&&
	f_1^{12\cdots k}(k,n)\sim \frac{n^42^n}{384} \quad \text{for} \quad k=3, \qquad
	f_1^{12\cdots k}(k,n)\sim \frac{n^43^n}{1944} \quad \text{for} \quad k=4,\\
	&&
	f_1^{12\cdots k}(k,n)\sim \frac{n^44^n}{6144} \quad \text{for} \quad k=5, \qquad
	f_1^{12\cdots k}(k,n)\sim \frac{n^45^n}{7500} \quad \text{for} \quad k=6.
	\feq
	\end{example}
	$\mbox{}$
	\par
	We refer to an edge of the associated graph starting and ending at the same state $\langle w \rangle$ as a
	\emph{loop} at $\langle w \rangle.$ It is easy to see that the graph does not have any cycles, besides perhaps loops (cf. \cite[p.~256]{HM}).
	Using similar arguments as in \cite{BM} (see Lemma~2.4 there), one can prove the following lemma.
	\begin{lemma}
	\label{lemloops}
	Let $d$ be the number of distinct letters in $v.$ Then for any $\langle u \rangle \in E(v,r,k),$
	the number of loops at $\langle u \rangle$ does not exceed $d-1$.
	Moreover, there are exactly $d-1$ loops at $\langle\epsilon \rangle$.
	\end{lemma}
	
	Recalling \eqref{GF}, the following lemma links the number of loops to the poles of the generating function $G_{k,n}^v(x),$ $x\in\cc,$ and hence to the asymptotic behavior of the sequence $g_r^v(k,n)$ as $n$ tends to infinity. The result follows directly from the identity in \eqref{a}
	and the transfer-matrix method \cite[Theorem 4.7.2]{Stanley1}. Given a matrix $A,$ denote by $A^{(i,j)}$ the matrix with row $i$ and column $j$ deleted.
	We have:
	\begin{lemma}
	\label{gf}
	Let $p=\#E(v,r,k)$ be the number of states in $\Au(v,r,k)$. Then the generating function $G_{k,n}^v(x)$ is given by
	\beq
	G_{k,n}^v(x) = \sum_{n\geq 0} g_r^v(k,n)x^n= \frac {\sum_{j=1}^p (-1)^{j+1}\det(I-xT^{(j,1)})}
	{\prod_{i=1}^p(1-\lambda_ix)}= \frac{ \det B(x)}{\prod_{i=1}^p(1-\lambda_ix)},
	\feq
	where $\lambda_i$ is the number of loops at state $s_i$, $T=T(v,r,k),$ and $B(x)$ is the
	matrix obtained by replacing the first column in $I-x T$ with a column of all ones.
	\end{lemma}
	\subsection{Stanley-Wilf type limits}
	\label{wftl}
	Throughout this section we assume that the number of distinct letters in the pattern $v\in [k]^\ell,$ namely $d,$ is greater than one.
	An interesting consequence of the results in Lemma~\ref{lemloops} and Lemma~\ref{gf} is the following theorem, which is the main result 
	of this section.	
	\par 
	Recall $f_r^v(k,n)$ and $g_r^v(k,n)$ from \eqref{gfr}.
	\begin{theorem}
	\label{thm:fnrlim_word}
	Assume that $d>1.$ Then for all $r \in \nn_0,$
	\beqn
	\label{lwd}
	\lim_{n\to \infty}\bigl(f_r^v(k, n)\bigr)^{\frac{1}{n}} = \lim_{n\to \infty}\bigl(g_r^v(k, n)\bigr)^{\frac{1}{n}}=d-1.
	\feqn
	\end{theorem}
	\begin{proof}
	By Lemma~\ref{gf}, the generating function $G_{r,k}^v(x)=\sum_{n\geq 0}g_r^v(k,n)x^n$ is a rational function in the complex plane $\cc.$ By Lemma~\ref{lemloops}, the smallest pole of $G_{r,k}^v(x)$ is $\frac{1}{d-1}.$
	Since the reciprocal of the smallest pole is the radius of convergence of the generating function \cite{analcombin}, we have
	\beq
	\limsup_{n\to \infty}(g_r^v(k, n))^{\frac{1}{n}} = d-1.
	\feq
	Since $f_r^v(k, n)\leq g_r^v(k, n)$, we conclude that
	\beq
	\limsup_{n\to \infty}(f_r^v(k, n))^{\frac{1}{n}}\leq d-1.
	\feq
	On the other hand, if $v\in [k]^\ell$ and a word $w\in[k]^*$ contains $v$ exactly $r$ times, then
	the concatenation $wu$ contains $v$ exactly $r$ times for any word $u\in[k]^*$ such that each letter of $u$ belongs to the set
	\beq
	\{1,2,\ldots,v_\ell-1, k-d+v_\ell+1,k-d+v_\ell+2,\ldots,k\},
	\feq
	where $v_\ell$ is the rightmost letter of $v$.
	Therefore, there exists a constant $c_r>0$ such that for all $n\in\nn,$
	\beq
	f_r^v(k,n)\geq c_r(v_\ell-1+k-k+d-v_\ell-1+1)^n=c_r(d-1)^n.
	\feq
	Hence,
	\beq
	\liminf_{n\to \infty}\bigl(f_r^v(k, n)\bigr)^{\frac{1}{n}}\geq d-1,
	\feq
	which completes the proof of the theorem.
	\end{proof}
	Note that the limit in \eqref{lwd} is independent of $r.$ It turns out that a similar result holds for the occurrence enumeration problem in permutations; see Theorem~\ref{thm:fnrlim} below. We remark that in the case of permutations, the structure of the dependence of the limit on the underlying pattern
	is considerably more complex than in \eqref{lwd} and is not yet completely understood \cite{wf1, fox, wf}. The theorem has an interesting implication 
	for the asymptotic behavior of the entropy of the random  variable $X_n=occ_v(W_n)$ with a random  $W_n,$ see Theorem~\ref{entropy} below for details.   
	\par
	A simple path in the graph representation of $Au(v,r,k)$ is a finite sequence of states $s_{j_0},\ldots,s_{j_q}$ in $E(v,r,k)$ such that
	$s_{i_0}=\langle \epsilon \rangle$ and for all $i=1,\ldots, q,$  we have $j_{i-1}<j_i$ and $s_{j_{i-1}}$ is connected to $s_{j_i}$ by a direct edge.
	The proof of the following partial refinement of Theorem~\ref{thm:fnrlim_word} follows that of Theorem~3.2 in \cite{BM} nearly verbatim, and therefore
	is omitted.	
	\begin{theorem}
	\label{word}
	Assume that $d>1.$ Let $M_r$ be the maximal number of states with $d-1$ loops
	in a simple path in $Au(v,r,k).$ Then  for any $r\geq 0,$ there exists a constant $C_r\in (0,\infty)$ and $K_r\geq 0$ such that
	\beqn
	\label{lw1}
	\lim_{n\to\infty} \frac{g_r^v(k, n)}{n^{M_r}(d-1)^n}=C_r\qquad \mbox{\rm  and}\qquad \lim_{n\to\infty} \frac{f_r^v(k, n)}{n^{M_r}(d-1)^n}=K_r.
	\feqn
	\end{theorem}

	Note that $M_r\geq 1$ by Lemma~\ref{lemloops}. Through investigating various patterns with $d>1,$ we observed $K_r>0.$ Nevertheless, we believe that the following is true:
	\begin{conj*}
	\label{Kconj}
	There exist $k,r\in\nn$ and a pattern $v\in [k]^*$ such that $d>1$ and $K_r$ in \eqref{lw1} is equal to zero.
	In that case, there exists $L_r\in\nn,$ $L_r<M_r,$ and $\witi K_r\in (0,\infty)$ such that
	$\lim_{n\to\infty} \frac{f_r^v(k, n)}{n^{L_r}(d-1)^n}=\witi K_r.$
	\end{conj*}
	It follows from the first limit identity in \eqref{lw1} that $M_r$ is a non-decreasing
	function of $r.$ If the previous conjecture is true, then $M_r$ is not always strictly increasing. We believe
	that the following is true:
	\begin{conj*}
	\label{mconj}
	\item [(a)] For any  $n,k\in\nn,$ and $v \in [k]^*$ with $d>1,$ $\lim_{r\to\infty} \frac{M_r}{r}$ exists and belongs to $(0,\infty).$
	\item [(b)] There exist $k\in\nn,$ a pattern $v\in [k]^*$ with $d>1,$ and an increasing sequence of integers $(r_n)_{n\in\nn},$ such that
	$M_{r_n}=M_{r_n-1}.$
	\end{conj*}
	We conclude this section with a remark that Theorems~\ref{thm:fnrlim_word} and~\ref{word} can be interpreted as large deviation
	estimates for $occ_v(w)$ when $w\in[k]^n$ is chosen at random, see Section~\ref{new} below for details.
	
	\subsection{Weak pattern avoidance}
	\label{weaka}
	In this section, we further investigate the asymptotic behavior of the sequence $(f_r^v(k,n))_{r\in \nn_0}.$ It turns
	out that the generating function of this sequence, as defined by \eqref{GF}, 
	can be linked to a natural concept of ``weak avoidance" that may be of independent interest. The weak avoidance is defined in a fashion similar to the notion
	of the weakly self-avoiding random walks \cite{saw}, namely by introducing a penalty for the non-avoidance rather
	than completely striking off the possibility of a pattern occurrence. 
	\par 	
	Formally speaking, for a pattern $v\in [k]^*$, we associate a sequence of penalty functions $c^v_{k,n}:[0,1]\to [0,k^n],$ $n\in\nn,$ as follows:
	\beqn
	\label{cw}
	c^v_{k,n}(x) = \sum_{w \in [k]^n}\, \prod_{1\leq j_1 < \cdots < j_\ell \leq n} \left( 1 + x U_{j_1,\cdots, j_\ell} (v,w) \right),
	\feqn
	where
	\beq
	U_{j_1, \cdots, j_\ell}(v, w) =
	\left\{
	\begin{array}{cl}
	-1 & \quad \text{if} \quad ( w_{j_q} \leq w_{j_r} ~ \Longleftrightarrow ~ v_q \leq v_r)\\
	0 & \quad \text{otherwise.}
	\end{array}
	\right.
	\feq
	It follows from \eqref{cw} that
	\beqn
	\label{th2words}
	c^v_{k,n}(x) =\sum_{w \in [k]^n} (1-x)^{occ_v(w)}=\sum_{r\geq 0} f_r^v(k,n)(1-x)^r.
	\feqn
	Thus $c^v_{k,n}(x) =F_{k,n}^v(1-x).$ According to the definition in \eqref{cw}, the function $c^v_{k,n}(x)$
	can be considered as a partition function counting the words in $[k]^n$ with weights penalizing occurrences of the pattern $v.$
	Note that $c^v_{k,n}(x)$ is a decreasing function of $x,$ $c_{k,n}^v(0) = k^n$ counts all words without discrimination,
	and on the opposite extreme $c_{k,n}^v(1) = f_0^v(k, n)$ counts only words avoiding the pattern entirely. The parameter $x\in [0,1]$ can
	be therefore interpreted as an intensity or strength of the pattern avoidance.
	\par
	The subsequent Section~\ref{new} is devoted to the study of the asymptotic behavior of the sequence $X_n=occ_v(W_n),$ $n\in\nn,$
	where $W_n=w_1,\cdots,w_n\in [k]^n$ and $w_i$ are i.\,i.\,d. random variables, each one distributed uniformly over $[k].$ The asymptotic behavior 
	of random variables $X_n=occ_v(W_n)$ in the case when the sequence $(w_i)_{i\in\nn}$ is drown at random from non-product
	probability measures on $[k]^\nn$ is beyond the topic of this paper and will be studied by the authors elsewhere. The only exception in this paper is Theorem~\ref{inv1} where, following a canonical construction in the theory of self-avoiding random walks \cite{saw}, we study $X_n$ in the case when $W_n$ is chosen at random according to the probability law
	\beqn
	\label{qword}
	\qq^{v,x}_{k,n}(A)=\frac{1}{c^v_{k,n}(x)}\sum_{w \in A}(1-x)^{occ_v(w)},\qquad A\subset [k]^n.
	\feqn
	Here $x$ is a parameter which ranges within the interval $[0,1].$
	Clearly, $\qq^{v,x}_{k,n}(\,\cdot\,)$ is not uniform on $[k]^n,$ it penalizes words $w$ with a non-zero $occ_v(w)$ by the factor $(1-x)^{occ_v(w)}$
	which depends on the parameter $x\in (0,1).$ This probability measure belongs to a general class of Boltzmann distributions intensively studied in statistical mechanics
	and combinatorics, cf. \cite{bolt}. In Theorem~\ref{inv1} we study $\qq^{v,x}_{k,n}$ in a certain small parameter regime where $x=x_n=o(1)$ decays fast, and consequently, $\qq^{v,x_n}_{k,n}$ can be considered as a perturbation of the uniform probability measure over $[k]^n.$
	\par
	We conclude this section with an analogue of Theorem~\ref{thm:fnrlim_word} for $c_{k,n}^v(x).$
	It follows from Theorem~\ref{thm:fnrlim_word} that for all $x\in[0,1],$
	\beqn
	\label{wlima}
	d-1\leq \liminf_{n\to \infty}(c_{k,n}^v(x))^{\frac{1}{n}} \leq  \limsup_{n\to \infty}(c_{k,n}^v(x))^{\frac{1}{n}}\leq k,
	\feqn
	where $d$ is the number of distinct letters in the pattern $v.$ We have:
	\begin{proposition}
	\label{th1words}
	Given a pattern $v\in[k]^\ell,$ $\lim_{n\to\infty}(c_{k,n}^v(x))^{\frac{1}{n}}$ exists and lies within the closed interval $[d-1,k]$ for all $x\in [0,1].$ 
	\end{proposition}
	\begin{proof}
	By the definition, for any $x\in [0,1],$ $w\in [k]^{\geq \ell},$ and an increasing sequence of indices $j_i,$ $1\leq i \leq \ell,$ we have		
	\beq
	0\leq 1+xU_{j_1,\cdots, j_\ell} (v,w)\leq 1.
	\feq
	Therefore, for any $n,m\in\nn$ and $x\in [0,1],$
	\beq
	&& c_{k,n+m}^v(x) = \sum_{w \in [k]^{n+m}}\, \prod_{1\leq j_1 < \cdots < j_\ell \leq n+m} \bigl( 1 + x U_{j_1,\cdots, j_\ell} (v,w) \bigr)
	\\
	&&
	\quad
	\leq  \sum_{w \in [k]^{n+m}}\, \prod_{1\leq j_1 < \cdots < j_\ell\leq m} \bigl( 1 + x U_{j_1,\cdots, j_\ell} (v,w) \bigr)
	\prod_{n+1\leq j_1 < \cdots < j_\ell \leq n+m}\bigl( 1 + x U_{j_1,\cdots, j_\ell} (v,w) \bigr) \\
			&&
	\quad = \Bigl(\sum_{w_1 \in [k]^m}\, \prod_{1\leq j_1 < \cdots < j_\ell \leq m} \bigl( 1 + x U_{j_1,\cdots, j_\ell} (v,w_1) \bigr)\Bigr)\times
	\\
	&&
	\quad \qquad \qquad \qquad
	\Bigl(\sum_{w_2 \in [k]^n}\, \prod_{1\leq j_1 < \cdots < j_\ell \leq n} \bigl( 1 + x U_{j_1,\cdots, j_\ell} (v,w_2) \bigr)\Bigr)
	\\
	&&
	\quad
	=c_{k,m}^v(x)c_{k,n}^v(x).
	\feq
	Hence $\log c_{k,n}^v(x),$ $n\in\nn,$ is a subadditive sequence, and the claim of the proposition follows from Fekete's subadditive lemma 
	and the estimates in \eqref{wlima}.
	\end{proof}
	\begin{example}
	\label{inv}
	Let us consider $v=21.$ In order to avoid the pattern $v$, the letters of a word $w\in [k]^n$ must be arranged in the non-decreasing order. Therefore, $f^{21}_0(k,n)=\binom{n+k-1}{k-1},$ the number of ways to write $n$ as a weak composition $n=a_1+\cdots+a_k,$ where $a_i\geq 0$ represents the number of occurrences of the letter $i\in [k]$ in a $k$-ary word of length $n.$ Furthermore,  by Theorem ~\ref{thm:fnrlim_word}, $\lim_{n\to\infty} \bigl( f^{21}_r(k,n)\bigr)^{1/n}=1$ for all integer $r\geq 0.$ Though a simple explicit expression for $f^{21}_r(k,n)$ is not known, a result on generating functions due to MacMahon (see, for instance, Theorem~3.6 in \cite{abook}) combined with \eqref{th2words} shows that for $x\in(0,1],$
	\beqn
	\label{cg}
	\nonumber
	c_{k,n}^{21}(x)&=&\sum_{r\geq 0} f^{21}_r(k,n)(1-x)^r=\, \sum_{\{a_j\geq 0:\, a_1+\ldots+a_k=n\}}\,\frac{\prod_{j=1}^n\bigl(1-(1-x)^j\bigr)}
	{\prod_{i=1}^k \prod_{j=1}^{a_i}\bigl(1-(1-x)^j\bigr)}
	\\
	\nonumber
	&\leq&
	\sum_{\{a_j\geq 0:\, a_1+\ldots+a_k=n\}}\,\frac{\prod_{j=1}^n\bigl(1-(1-x)^j\bigr)}
	{\bigl(1-(1-x)\bigr)^{k-1}\prod_{j=1}^{n-k+1}\bigl(1-(1-x)^j\bigr)}
	\\
	&<&
	\frac{1}{x^{k-1}}\binom{n+k-1}{k-1} .
	\feqn
	The first inequality in \eqref{cg} follows readily from the fact that
	\beq
	(1-s^i)(1-s^j)\geq (1-s^{i-1})(1-s^{j+1})\qquad \forall\,s\in (0,1)
	\feq
	as long as $i\leq j+1.$ Combining \eqref{cg} with the trivial inequality $c^{21}_{k,n}> f^{21}_0(k,n),$ we obtain that $\binom{n+k-1}{k-1}< c^{21}_{k,n}(x)<\frac{1}{x^{k-1}}\binom{n+k-1}{k-1}$ for all $x\in (0,1).$  Remark that a straightforward improvement
	of the lower bound for $c^{21}_{k,n}(x)$ is
	\beq
	c_{k,n}^{21}(x)&=&\,\sum_{\{a_j\geq 0:\, a_1+\ldots+a_k=n\}}\,\frac{\prod_{j=1}^n\bigl(1-(1-x)^j\bigr)}
	{\prod_{i=1}^k \prod_{j=1}^{a_i}\bigl(1-(1-x)^j\bigr)}
	\\
	&\geq&
	\sum_{\{a_j\geq 0:\, a_1+\ldots+a_k=n\}}\,\frac{\prod_{j=1}^n\bigl(1-(1-x)^j\bigr)}
	{\Bigl(\prod_{j=1}^{\lfloor n/k \rfloor+1}\bigl(1-(1-x)^j\bigr)\Bigr)^k}
	\\
	&\geq &\binom{n+k-1}{k-1}
	\,\frac{\prod_{j=1}^n\bigl(1-(1-x)^j\bigr)}{\Bigl(\prod_{j=1}^{\lfloor n/k \rfloor+1}\bigl(1-(1-x)^j\bigr)\Bigr)^k},
	\feq
	where $\lfloor a \rfloor$ denotes the integer part of $a\in\rr.$ Combining this lower bound with \eqref{cg}, we obtain that for all $x\in(0,1),$
	\beqn
	\label{poi}
	 \frac{\left(\varphi(1-x)\right)^{k-1}}{(k-1)!}\leq \liminf_{n\to\infty} \frac{c_{k,n}^{21}(x)}{n^{k-1}}\leq\limsup_{n\to\infty} \frac{c_{k,n}^{21}(x)}{n^{k-1}}\leq \frac{1}{x^{k-1}(k-1)!},
	\feqn
	where $\varphi(x)$ is the Euler generating function $\prod_{j=1}^\infty \frac{1}{1-x^j}.$ Notice that the lower and upper bounds in \eqref{poi} match asymptotically when $x\to 1.$
	\end{example}
	
	\subsection{Random words}
	\label{new}
	Let $(w_i)_{i\in\nn}$ be a sequence of independent random variables, each distributed uniformly on $[k]$ and let $v\in [k]^\ell$ be a word pattern,
	$\ell\geq 2.$ Denote $W_n=w_1w_2\cdots w_n\in[k]^n$ for $n\in\nn,$ and let $W=w_1w_2\cdots $ be the infinite string compound from the successive letters in the sequence.
	In this section we study the asymptotic behavior of the random variable	$X_n=occ_v(W_n)$. Note that for all $r\in\nn_0,$
	\beq
	P_n(r):=P(X_n=r)=\frac{1}{k^n}f_r^v(k,n).
	\feq 
	We start with a corollary to Theorem~\ref{thm:fnrlim_word} that is concerned 
	with the asymptotic behavior  of the information entropy of $X_n,$ when $n$ tends to infinity. Let 
	\beq
	H_{k,v}(n)=-\sum_{r\geq 0} P_n(r)\log P_n(r)
	\feq
	be the entropy of the random variable $X_n.$ The following theorem shows that $H_{k,v}(n)$ grows linearly with $n$ 
	and gives the exact rate of growth for an arbitrary pattern $v$ with $d>1.$   
	\begin{theorem}
	\label{entropy}
	Assume that $d>1.$ Then, 
	\beq 
	\lim_{n\to\infty} \frac{H_{k,v}(n)}{n} =\log\frac{k}{d-1}.
	\feq
	\end{theorem}
	\begin{proof}
	We have 
	\beq
	H_{k,v}(n)&=&-\sum_{r\geq 0} P_n(r)\log P_n(r)=-\frac{1}{k^n}\sum_{r\geq 0} f_r^v(k,n)\bigl(\log f_r^v(k,n)-n\log k\bigr)
	\\
	&=&
	n\log k-\sum_{r\geq 0} P_n(r) \log f_r^v(k,n).
	\feq
	Thus 
	\beq
	\frac{H_{k,v}(n)}{n} =\log k-\sum_{r\geq 0} P_n(r)\frac{\log f_r^v(k,n)}{n},
	\feq
	and the result follows from Theorem~\ref{thm:fnrlim_word} and a discrete version of the bounded convergence theorem. 
	\end{proof}
	\par       
	Our next result is a central limit theorem for $X_n$ which asserts that, as $n$ tends to infinity,
	$X_n$ is highly concentrated at $E(X_n)=\binom{n}{\ell}\binom{k}{d}\frac{1}{k^\ell}$ with standard deviation of order $\frac{1}{\sqrt{n}}E(X_n).$ The fact that, exactly as in the classical case of partial  sums of i.\,i.\,d. variables, typical fluctuations of $X_n$ are of order $\frac{1}{\sqrt{n}}E(X_n)$ will be often exploited in the rest of this section.  The proof follows closely that of Theorem~2 in \cite{bona}, a similar CLT for pattern occurrences in permutations. It is based on an application of a general CLT for dependent variables due to \cite{janson}, and hence, it relies on an accurate estimation of $\mbox{VAR}(X_n).$  Given the variance estimate and a general result in \cite{becite}, the CLT can be strengthen to a Berry-Esseen type result providing the classical $O(n^{-1/2})$ rate of convergence, see Corollary~\ref{bess} below.   
	\begin{theorem}
	\label{wclt}
	Let $\mu_n=E(X_n)$ and $\sigma_n=\sqrt{\text{VAR}(X_n)}.$ Then $\mu_n=\binom{n}{\ell}\binom{k}{d}\frac{1}{k^\ell},$
	$\sigma_n=\Theta\bigl(\frac{\mu_n}{\sqrt{n}}\bigr),$ and $\frac{X_n-\mu_n}{\sigma_n}$ converges in distribution, as $n\to\infty,$
	to a standard normal random variable.
	\end{theorem}
	\begin{proof}
	There are $\binom{n}{\ell}$ ways to choose $\ell$ indexes $j_1<\cdots j_\ell$ out of $n$ possibilities.
	We refer to these ordered $\ell$-tuples as $\ell$-subintervals of $[n].$ Enumerate these subintervals in an arbitrary manner,
	and let $I_j,$ $j=1,\ldots,\binom{n}{\ell},$ denote the $j$-th subinterval. Let $X_{n,j}$ be the indicator of the event
	that the pattern occurs at $j$-th subinterval. \par
	First, we will compute $E(X_n).$ Given that
	\beq
	E(X_{n,j})=\frac{1}{k^\ell}\binom{k}{d},\qquad 1\leq j\leq \binom{n}{\ell},
	\feq
	and $X_n=\sum_{j=1}^{\binom{n}{\ell}} X_{n,j},$ we have
	\beqn
	\label{ex}
	E(X_n)=\binom{n}{\ell}\binom{k}{d}\frac{1}{k^\ell}.
	\feqn
	
	Next, we will estimate $\text{VAR}(X_n).$  To that end, we rewrite $X^2_n$ as follows
	\beq
	X_n^2=\sum_{1\leq j,m \leq \binom{n}{\ell}} X_{n,j}X_{n,m}=\sum_{s=0}^\ell A_s, \feq where \beq 
	A_s:=\sum_{\{j,m:|I_j\cap I_m|=s\}}X_{n,j}X_{n,m}.
	\feq
	In what follows, we will adopt the proof strategy of \cite{bona} and estimate $E(A_s)$ separately for different values of the parameter $s.$ For $s=0$ the exact value is
	\beq
	E(A_0) = \binom{n}{\ell}\binom{n-\ell}{\ell}\frac{1}{k^{2\ell}}\binom{k}{d}^2
	\feq
	where we used the fact that for two intervals $I_j$ and $I_m$ with no overlap
	\beq
	E(X_{n,j}X_{n,m})=\frac{1}{k^{2\ell}}\binom{k}{d}^2.
	\feq
	If $A_0$ would be the only terms contributing to the variance of $X_n,$ its entire contribution combined with the term $-\bigl[E(X_n)\bigr]^2$
	would amount to (cf. formulas (9) and (10) in \cite{bona})
	\beqn
	\nonumber 
	E(A_0)-[E(X_n)]^2&=&\binom{n}{\ell}\binom{n-\ell}{\ell}\frac{1}{k^{2\ell}}\binom{k}{d}^2-\left(\binom{n}{\ell}\frac{1}{k^\ell}\binom{k}{d}\right)^2
	\\
	\nonumber
	&=&-n^{2\ell-1}\frac{\ell^2}{(\ell!)^2k^{2\ell}}\binom{k}{d}^2+O(n^{2\ell-2}) \\ 
	\label{amount}
	&=& -\Theta(n^{2\ell-1})
	\feqn
	To finish the estimate on the variance we need to provide estimates on $A_s$ when $s\neq 0.$ More specifically,  when $s=1$ we give an accurate estimate, and for $s\geq 2$ a crude estimate will suffice for our purpose. More specifically, we will show that $E(A_s)=\Theta(n^{2\ell-s}),$ and while $E(A_0)-[E(X_n)]^2$ is negative, $E(A_0)+E(A_1)-[E(X_n)]^2=\Theta(n^{2\ell-1})$ which gives the necessary estimate for the variance.
	\par
	Case I: $s=1$. Consider the sum of the terms $E(X_{n,i}X_{n,j})$ over the pairs of intervals that overlap exactly at one place. The summation of these terms is 
	\beqn
	\label{dvk}
	E(A_1)=\sum_{\{j,m:|I_j\cap I_m|=1\}}E(X_{n,j}X_{n,m})=\binom{n}{2\ell-1}D_{k,v}
	=\Theta(n^{2\ell-1})D_{k,v},
	\feqn
	where
	\beq
	D_{k,v}&\geq& \frac{1}{k^{2\ell-1}}\sum_{i=0}^{\ell-1} \binom{2i}{i}\binom{2\ell-2-2i}{\ell-1-i}\cdot \min_{1\leq p\leq d}\sum_{t=p}^{k-d+p} \left\{\binom{t-1}{p-1}\binom{k-t}{d-p}\right\}^2,
	\feq
	with two words occupying the intervals $I_j$ and $I_m$ overlap over the $(i+1)$-th letter of each, and $v_{i+1}$ being the $p$-th highest letter (among
	the distinct possibilities $1,\ldots,d$) in the pattern $v.$ To obtain the lower bound for $D_{k,v}$ we will only consider the case when the common letter is the $(i+1)$-th letter for some $i\in\{0,\ldots,\ell-1\}$ in both intervals. Once the joint location of $I_j$ and $I_m$ is  chosen, we have in total $k^{2\ell-1}$ possibilities to choose the corresponding letters. We have to fill $2i$ locations before and and $2\ell-2-2i$ locations after the common letter. The term $\binom{2i}{i}\binom{2\ell-2-2i}{\ell-1-i}$ is the number of possibilities to designate $\ell-1$ of the remaining $2\ell-2$ locations to be occupied by letters of the interval $I_m.$ Assuming that for given $p$ and $t$ the common letter for $I_m$ and $I_j$ is $t\in [p,p+1,\ldots, k-(d-p)],$ we observe that we have $\binom{t-1}{p-1}\binom{k-t}{d-p}$ possibilities to choose $d$ distinct letters from $[k].$ 
	\par
	We remark that
	\beq
	\frac{1}{k^{2\ell-1}}\sum_{i=0}^{\ell-1} \binom{2i}{i}\binom{2\ell-2-2i}{\ell-1-i}=
	\frac{1}{k^{2\ell-1}}\binom{2\ell-2}{\ell-1}\sum_{i=0}^{\ell-1} \frac{\binom{\ell-1}{i}\binom{\ell-1}{i}}{\binom{2\ell-2}{2i}}
	\geq
	\frac{2}{k^{2\ell-1}}\binom{2\ell-2}{\ell-1},
	\feq
	where the inequality is obtained by enumerating the terms with $i=0$ and $i=\ell-1$ only.
	\par
	Furthermore, \beq
	\sum_{j=p}^{k-d+p} \left\{\binom{j-1}{p-1}\binom{k-j}{d-p}\right\}^2&\geq& (k-d+1)\left\{\sum_{j=p}^{k-d+p} \binom{j-1}{p-1}\binom{k-j}{d-p}\right\}^2
	\\
	&=&
	(k-d+1)\binom{k}{d}^2\geq \binom{k}{d}^2,
	\feq
	where we used Cauchy-Schwartz inequality in the first inequality and a variation of the Chu-Vandermonde identity stated as 
	\beq
	\sum_{j=p}^{k-d+p} \binom{j-1}{p-1}\binom{k-j}{d-p}=\binom{k}{d}.
	\feq
	This identity can be justified as follows: in order to choose $d$ distinct letters from $[k]$ we can first choose the $p$-th largest element among those $d$ letters, call it $j,$
	from the interval $[p,k-d+p],$ then $p-1$ letters from the interval $[1,j-1]$ and $d-p$ letters from the interval $[j+1,k].$ Collecting all the estimates together, we obtain that 
	\beqn 
	\label{a11}
	E(A_1)\geq \binom{k}{d}\left\{\binom{n}{2\ell-1}\frac{2}{k^{2\ell-1}}\right\}
	\binom{2\ell-2}{\ell-1}. 
	\feqn
	Case II: $s>1$. Furthermore, extending \eqref{dvk} to
	\beqn
	\label{a14}
	E(A_s)=\sum_{\{j,m:|I_j\cap I_m|=i\}}E(X_{n,j}X_{n,m})=\binom{n}{2\ell-i}D_{k,v}^{(i)}=\Theta(n^{2\ell-i}),
	\feqn
	where $D_{k,v}^{(i)}>0$ are strictly positive constants whose value depends on $k$ and $v$ only (but not on $n$). 
	\par
	Having in hand the above estimates for $E(A_n)$ we can now evaluate the variance of $X_n.$ Taking into the account \eqref{amount}, \eqref{a11}, and \eqref{a14}, we obtain that
	\beqn
	\nonumber
	\text{VAR}(X_n)&\geq& \binom{k}{d}\left\{\binom{n}{2\ell-1}\frac{2}{k^{2\ell-1}}
	\binom{2\ell-2}{\ell-1}-n^{2\ell-1}\frac{\ell^2}{(\ell!)^2k^{2\ell}}+O(n^{2\ell-2})\right\}
	\\
	&=&
	\label{varx}
	\delta_{k,v}n^{2\ell-1}+O(n^{2\ell-2}),
	\feqn
	where
	\beqn
	\nonumber
	\delta_{k,v}&=&\binom{k}{d}\left\{\frac{2}{k^{2\ell-1}(2\ell-1)((\ell-1)!)^2}-\frac{\ell^2}{(\ell!)^2k^{2\ell}}\right\}=
	\binom{k}{d}\frac{\ell^2}{k^{2\ell}(\ell!)^2}
	\Bigl(\frac{2k}{2\ell-1}-1\Bigr)
	\\
	\label{dex}
	&\geq&
	\frac{\ell^2}{k^{2\ell}(\ell!)^2}\Bigl(\frac{2\ell}{2\ell-1}-1\Bigr)=\frac{\ell^2}{k^{2\ell}(\ell!)^2(2\ell-1)}>0.
	\feqn
	Finally, by virtue of \eqref{ex}, the following limit exists and is strictly positive:
	\beqn
	\label{jvk}
	J_{k,v}:=\lim_{n\to\infty}\frac{\mu_n}{\sigma_n\sqrt{n}}>0,
	\feqn
	and therefore, the remainder of the proof is a straightforward application of Theorem~2 in \cite{janson} to the random variables $X_{n,i},$ and can be
	carried as in \cite{bona} verbatim.
	\end{proof}
	\begin{remark}
	\label{rclt}
	A central limit theorem for multisets closely related to Theorem~\ref{wclt} can be found in \cite{feray},
	see also references therein for earlier versions. Let $a_{i,n}\geq 0$ represent the number of occurrences of the letter $i\in [k]$ in the random word $W_n,$
	and denote by $A_n$ the random vector $(a_{1,n},\ldots,a_{k,n}).$ The CLT for $W_n$ in \cite{feray} can be stated as a limit theorem for the random variable $\frac{X_n-\witi \mu_n}{\witi \sigma_n}$ under the conditional measure $P(\cdot\,|A_n).$
	The main difference with Theorem~\ref{wclt} is that the scaling factors $\witi \mu_n=\witi \mu_n(A_n)$ and $\witi\sigma_n=\witi\sigma_n(A_n)$
	are random in that they depend on the vector $A_n.$ The relation of Theorem~\ref{wclt} to the CLT in \cite{feray} thus resembles the one between the so called annealed (average) and quenched limit theorems in the theory of random motion in a random media, see, for instance, \cite{notes}. In particular, $\sigma_n^2=E\bigl(\witi \sigma_n^2\bigr)+ \widehat \sigma_n^2,$ where $\sigma_n^2$ is the ``annealed" variance that appears in the statement of Theorem~\ref{wclt} whereas the term $\widehat \sigma_n^2$ describes fluctuations of the ``random environment" $A_n.$
	\end{remark}
	Our next result is a Berry-Esseen type bound for the convergence rate of the above CLT. The bound is a direct implication of Theorem~2.2 in \cite{becite}, along with the estimates in \eqref{varx}, \eqref{dex}, and the following modification of \eqref{amount}:
	\beqn
	\label{amount1}
	\Delta_n=\binom{n}{\ell}-\binom{n-\ell}{\ell}-1
	=n^{2\ell-1}\frac{\ell^2}{\ell!}+O(n^{2\ell-2}).
	\feqn
	Here $\Delta_n$ is the number of random indicators $X_{n,i}$ that are independent of $X_{n,i^*},$ an indicator with a given index $1\leq i^* \leq \binom{n}{\ell}.$ Let $\Phi(x)=\sqrt{\frac{1}{2\pi}} \int_{-\infty}^xe^{-\frac{x^2}{2}}\,dx,$ $x\in\rr,$ denote the distribution function of the standard normal variable. We have:
	\begin{corollary}
	\label{bess}
	In the notation of Theorem~\ref{wclt},
	\beq
	\sup_{x\in\rr}\,\Bigl|P\Bigl(\frac{X_n-\mu_n}{\sigma_n}\leq x\Bigl)-\Phi(x)\Bigr|\leq
	\frac{k^{\ell+2} \ell!}{k!}\,\sqrt{\frac{\ell}{\pi n}} +O(n^{-3/2})+O(n^{-\ell/2}) .
	\feq
	\end{corollary}
	Remark that the classical Berry-Esseen bound for the rate of convergence of the CLT for partial sums of i.\,i.\,d.
	random variables is of order $n^{-1/2},$ thus the above bound is asymptotically optimal up to a constant.
	\par
	Theorem~\ref{wclt} implies a weak law of large numbers for $X_n$ and asserts that a typical deviation of $X_n$
	from $E(X_n)$ is of order $\frac{1}{\sqrt{n}}E(X_n).$ The main purpose of the following Chernoff type bounds is to estimate the probability
	of large deviations, namely the ones of the order of magnitude $E(X_n).$ The result is merely an instance of Corollary~2.6 in \cite{LDPJ} formulated using the notation of Theorem~\ref{wclt}.
	\begin{corollary}
	\label{wsldp}
	For any $t\geq 0,$
	\beq
	P(X_n\geq \mu_n+t)\leq \exp\Bigl\{-\frac{t^2(1-\Delta_n/4K_n)}{2\Delta_n(\mu_n+t/3)(1-\mu_n/K_n)}\Bigr\}
	\feq
	and
	\beq
	P(X_n\leq \mu_n-t) \leq \exp\Bigl\{-\frac{t^2(1-\Delta_n/4K_n)}{2\Delta_n\mu_n}\Bigr\},
	\feq
	where $\Delta_n$ is introduced in \eqref{amount1} and $K_n=\binom{n}{\ell}.$
	\end{corollary}
	We will now state a direct consequence of Theorem~\ref{wclt} in terms of the weak avoidance penalty function $c^v_{k,n}(x).$
	Our main motivation for including this result is the subsequent Theorem~\ref{inv1}. Recall the notation of Theorem~\ref{wclt}.
	\begin{lemma}
	\label{cvclt}
	\item [(a)] Let $(\theta_n)_{n\in\nn}$ be a sequence of positive reals such that $\lim_{n\to\infty} \theta_n=+\infty$ and
	$\lim_{n\to\infty} \frac{\mu_n}{\theta_n}=\gamma$ for some $\gamma\in [0,+\infty).$ Then, the following holds for any constant
	$t\in\rr:$
	\beqn
	\label{j}
	\lim_{n\to\infty} \frac{\theta_n}{\mu_n}\log E\bigl(e^{\frac{tX_n}{\theta_n}}\bigr)
	=\lim_{n\to\infty} \frac{\theta_n}{\mu_n}\log E\Bigl[\Bigl(1+\frac{t}{\theta_n}\Bigr)^{X_n}\Bigr]= t
	\feqn
	\item [(b)] The following holds for any constant $t\in\rr:$
	\beqn
	\label{j3}
	\lim_{n\to\infty} \frac{1}{\sqrt{n}}\log E\bigl(e^{\frac{tX_n\sqrt{n}}{n^\ell}}\bigr)=\lim_{n\to\infty} \frac{1}{\sqrt{n}}\log E\Bigl[\Bigl(1+\frac{t\sqrt{n}}{n^\ell}\Bigr)^{X_n}\Bigr]=J_{k,v}t,
	\feqn
	where $J_{k,v}$ are strictly positive constants introduced in \eqref{jvk}.
	\end{lemma}
	\begin{proof}
	Observe that all the expectations in the statement of the lemma are well-defined for all $t\in\rr$
	because $1\leq X_n\leq \binom{n}{l}.$ Let $s=e^t.$ We will use the parameter $s$ so defined in both parts, (a) and (b), of the proof.
	\item [(a)]  We will consider separately two cases, $\gamma=0$ and $\gamma\in (0,\infty).$
	\item [] Case I: $\gamma=0.$ Using the second-order Taylor series with the remainder in the Lagrange form
	\beqn
	\label{lf}
	e^y=1+y+\frac{e^{y^*}y^2}{2},\quad \mbox{\rm with}~y=\frac{X_n t}{\theta_n}>0,\, |y^*|\in [0,|y|],
	\feqn
	we obtain:
	\beq
	&&
	\frac{\theta_n}{\mu_n}\log \bigl\{E(e^{tX_n/\theta_n})\bigr\}=
	\frac{\theta_n}{\mu_n}\log \Bigl\{E\Bigl(1+\frac{t X_n}{\theta_n}+e^{t_n^*X_n/\theta_n}\frac{(X_n t)^2}{2(\theta_n)^2}\Bigr)\Bigr\}
	\feq
	for some random (because of the dependence on $X_n$) $t_n^*\in [0,|t|].$
	Note that in view of \eqref{ex} and the condition $\lim_{n\to\infty}\frac{\mu_n}{\theta_n}<\infty,$ with probability one,
	\beqn
	\label{mkv}
	\sup_{n\in \nn}e^{t_n^*X_n/\theta_n}\leq \sup_{n\in \nn}e^{|t|\binom{n}{\ell}/\theta_n}<M_{k,v}(t)
	\feqn
	for some (deterministic) constant $M_{k,v}(t)>0$ which depends on the parameters $k,v$ and $t.$
	Furthermore, by Theorem~\ref{wclt}, $E(X_n^2)= \sigma_n^2+\mu_n^2\sim \mu_n^2.$ Therefore,
	\beqn
	\label{uta}
	\lim_{n\to\infty} \frac{\theta_n}{\mu_n}\log E(e^{tX_n/\theta_n})=t.
	\feqn
	Recall the constant $M_{k,v}(t)$ in \eqref{mkv}. For any $r\in \nn,$ we have
	\beq
	\Bigl|e^{rt/\theta_n}-\Bigl(1+\frac{t}{\theta_n}\Bigr)^r\Bigr|&=&
	\Bigl|e^{rt/\theta_n}-\Bigl(1+\frac{t}{\theta_n}\Bigr)^r\Bigr|\leq rM_{k,v}(t)
	\Bigl|e^{\frac{t}{\theta_n}}-1-\frac{t}{\theta_n}\Bigr|
	\\
	&\leq& \frac{M_{k,v}(t)t^2}{2} \frac{r}{\theta_n^2},
	\feq
	where we used the mean-value theorem applied to the function $f(y)=y^r$ in the first step and \eqref{lf} in the second one. Since,
	\beq
	\frac{\theta_n}{\mu_n}\cdot \Bigl|\frac{1}{k^n}\sum_{r\geq 0}f_r^v(k,n)\frac{r}{\theta_n^2}\Bigr|
	=\frac{1}{\theta_n}
	\to 0~\mbox{\rm as $n$ tends to $0$},
	\feq
	we get \eqref{j} for $\gamma=0$ by utilizing \eqref{uta}.
	\item [] Case II: $\gamma\in(0,\infty).$  In this case, \eqref{uta} follows directly from the law of large numbers $X_n/\mu_n\Rightarrow 1$
	in probability, as $n\to\infty,$ which is implied by Theorem~\ref{wclt}. The rest of the proof of \eqref{j} is the same as in Case I.
	\\
	$\mbox{}$
	\item [(b)] By Theorem~\ref{wclt}, for any $t\in\rr$ we have:
	\beq
	\lim_{n\to\infty} e^{-t\mu_n/\sigma_n}E(e^{tX_n/\sigma_n})=e^{\frac{t^2}{2}}.
	\feq
	The convergence of the moment generating functions of $\frac{X_n-\mu_n}{\sigma_n}$ can be verified using, for instance,  a general Theorem~3 in \cite{ccurtiss}, it is also transparent from the proofs in \cite{janson}. It follows that
	\beq
	\lim_{n\to\infty} \Bigl(-\frac{\mu_n t}{\sigma_n}+\log \big\{E(e^{tX_n/\sigma_n})\bigr\}\Bigr)=\frac{t^2}{2},
	\feq
	and hence
	\beq
	\lim_{n\to\infty} \frac{1}{\sqrt{n}}\log \big\{E(e^{tX_n/\sigma_n})\bigr\}=J_{k,v}t.
	\feq
	The last formula is an analogue of \eqref{uta} in part (a) and plays a similar role, the remainder of the argument is similar to its counterpart in (a).
	\end{proof}
	Recall $\qq^{v,x}_{k,n}$ from \eqref{qword} and let $\ee^{v,x}_{k,n}$ denote the expectation with respect to $\qq^{v,x}_{k,n}.$
	Then for any $z>0$ and $x\in (0,1)$ we have
	\beqn
	\label{eqq}
	\ee^{v,x}_{k,n}(z^{X_n})=\frac{1}{c^v_{k,n}(x)}\sum_{w \in A}(1-x)^{occ_v(w)}z^{occ_v(w)}=\frac{E[(z(1-x))^{X_n}]}{E[(1-x)^{X_n}]}.
	\feqn
	Two interesting regimes in this model arise when it is assumed that $x=x_n$ depends on $n$ and either $x_n=o(1)$ or $1-x_n=o(1).$ 
	Both the regimes can be considered as a perturbation of a uniform distribution, over $S_n$ in the former case and over the pattern-avoiding 
	set $\{w\in [k]^n:occ_v(w)=0\}$ in the latter. In the context of permutations, similar regimes for the particular case when the pattern is the inversion $21,$ were recently studied in \cite{nora, mall,starr}. In view of \eqref{eqq}, Lemma~\ref{cvclt} implies the following:
	\begin{theorem}
	\label{inv1}
	\item [(a)] Let $(\theta_n)_{n\in\nn}$ and $(\rho_n)_{n\in\nn}$ be two sequences of positive reals such that
	$\lim_{n\to\infty} \theta_n=+\infty,$ $\lim_{n\to\infty} \frac{\mu_n}{\theta_n}=\lambda$ for some $\lambda\in [0,+\infty),$ and
	$\lim_{n\to\infty} \frac{\theta_n}{\rho_n}=\alpha$ for some $\alpha\in [0,+\infty).$ Then the following holds for any $t\in\rr:$
	\beq
	\lim_{n\to\infty} \frac{\theta_n}{\mu_n}\log\ee^{v,\frac{1}{\rho_n}}_{k,n}(e^{\frac{tX_n}{\theta_n}})= t.
	\feq
	In particular, by virtue of \eqref{ex},
	\beqn
	\label{j1}
	\lim_{n\to\infty} \ee^{v,\frac{1}{\rho_n}}_{k,n}(e^{\frac{tX_n}{n^\ell}})=\exp\Bigl[\frac{t}{k^\ell\ell !}\binom{k}{d}\Bigr]
	\feqn
	if $\lim_{n\to\infty} \frac{n^\ell}{\rho_n} \in [0,+\infty).$
	\item [(b)] The following holds for any $t\in\rr$ and a sequence of positive reals $(\rho_n)_{n\in\nn} $ such that
	$\lim_{n\to\infty} \frac{n^\ell}{\rho_n\sqrt{n}}=\beta$ for some $\beta\in [0,+\infty):$
	\beqn
	\label{j5}
	\lim_{n\to\infty} \frac{1}{\sqrt{n}}\log\ee^{v,\frac{1}{\rho_n}}_{k,n}(e^{\frac{tX_n\sqrt{n}}{n^\ell}})=J_{k,v}t,
	\feqn
	where $J_{k,v}$ are strictly positive constants introduced in \eqref{jvk}.
	\item [(c)] The following holds for any $t\in\rr$ and a sequence of positive reals $(\rho_n)_{n\in\nn} $ such that
	\beq
	\lim_{n\to\infty} \frac{n^\ell}{\rho_n}=\gamma
	\feq 
	for some $\gamma\in [0,+\infty):$
	\begin{itemize}
	\item [(i)]  We have:
	\beqn
	\label{e1} 
	\lim_{n\to\infty} \ee^{v,\frac{1}{\rho_n}}_{k,n}\Bigl(\frac{X_n}{n^\ell}\Bigr)=\frac{1}{k^\ell\ell !}\binom{k}{d}.
	\feqn
	\item [(ii)] Let $\qq_n(r)=\qq^{v,\frac{1}{\rho_n}}_{k,n}(X_n=r)$ and 
	\beq
	\hh_n=-\sum_{r\geq 0}\qq_n(r)\log \qq_n(r)
	\feq
	be the entropy of $X_n$ under the law $\qq^{v,\frac{1}{\rho_n}}_{k,n}.$ Then 
	\beq
	\lim_{n\to\infty} \frac{\hh_n}{n}=\log\frac{k}{d-1}+\frac{\gamma}{k^\ell\ell !}\binom{k}{d}.
	\feq
	\end{itemize}
	\end{theorem}
	\begin{proof}
	For part (a), plug $x=\frac{1}{\rho_n}$ and $z=e^{t/\theta_n}$ into \eqref{eqq} and use \eqref{j}. For part (b), substitute $z=e^{\frac{t\sqrt{n}}{n^\ell}}$  and use \eqref{j3}. Part (i) in (c) follows then from the bounded convergence theorem and \eqref{j1} which implies that the distribution of $\frac{X_n}{n^\ell}$ under the law $\ee^{v,\frac{1}{\rho_n}}_{k,n}$  converges to the degenerate distribution at $\frac{1}{k^\ell\ell !}\binom{k}{d}.$ 
	Finally, 
	\beq
	\hh_n&=&-\sum_{r\geq 0} \pp_n(r)\log \frac{f_r^v(k,n)(1-x)^r}{c_{k,n}^v(x)}
	\\
	&=&
	-\sum_{r\geq 0} \pp_n(r)\log f_r^v(k,n) -\ee^{v,\frac{1}{\rho_n}}_{k,n}(X_n)\log(1-\rho_n^{-1})+\log c_{k,n}^v(\rho_n^{-1}),
	\feq
	which implies the claim in (ii) of part (c). Indeed, $\frac{1}{n}\sum_{r\geq 0} \pp_n(r)\log f_r^v(k,n)$ converges to $\log(d-1)$ by  Theorem~\ref{thm:fnrlim_word} and a discrete version of the bounded convergence theorem, 
	$\frac{1}{n}\ee^{v,\frac{1}{\rho_n}}_{k,n}(X_n)\log(1-\rho_n^{-1}) \sim \frac{\mu_n}{\rho_n}$ by \eqref{j1}, and 
	$\frac{1}{n}\log c_{k,n}^v(\rho_n^{-1})$ converges to $\log k$ by virtue of \eqref{j}. The proof of the theorem is complete. 
	\end{proof}
	The results in Theorem~\ref{inv1} shed some light on the asymptotic behavior of $X_n$ under $\qq^{v,x_n}_{k,n}$ for $x_n=o(1).$
	More specifically, the corollary suggests that the intensity sequence $x_n=1/\rho_n$ with $\rho_n$ which is at least $\Theta(\mu_n)$
	yields a perturbative ``light avoidance regime" in that the results in Lemma~\ref{cvclt} and Theorem~\ref{inv1} formally correspond to their counterparts in the corollary with $\rho_n=+\infty.$ In particular, \eqref{j1} shows that $\mu_n$ remains the proper scaling for $X_n$ for any $x_n$ in this regime, namely the distribution of $X_n/\mu_n$ under $\qq^{v,x_n}_{k,n}$ converges to that of the constant one as $n\to\infty.$ Furthermore, by the G\"{a}rtner-Ellis theorem \cite{ldpbook},
	the result in \eqref{j5} for moment generating functions implies Corollary~\ref{wldp} given below.     
	\begin{corollary}
	\label{wldp}
	Let $\rho_n$ be as defined in the statement of part (b) of Theorem~\ref{inv1}. Then the following holds for any Borel set $B\subset \rr:$
	\beq
	\lim_{n\to\infty} \frac{1}{\sqrt{n}}\log\qq^{v,\frac{1}{\rho_n}}_{k,n}\Bigl(\frac{X_n}{n^\ell}\in B\Bigr)=-\infty.
	\feq
	\end{corollary}
	It is reasonable to expect that a large deviation principle for $X_n/n^\ell$ under $\qq^{v,\frac{1}{\rho_n}}_{k,n}$ holds
	with a finite rate function and with respect to the usual scaling sequence $n$ rather than $\sqrt{n}$ (in our context, cf. Corollary~\ref{wsldp} where
	$\frac{\mu_n^2}{\Delta_n \mu_n}=\frac{\mu_n}{\Delta_n}=\Theta(n)$). However, proving such a result would be beyond the reach of methods
	we employed in this section.
	\par
	We conclude the section with another corollary to Theorem~\ref{wclt}, a limit theorem that concerns with a Poisson approximation of  $X_n$ in the case when $k=k_n$ is a rapidly enough increasing function of $n.$ The result is an analogue for random words of \cite[Theorem~3.1]{crane} for random permutations.
	The proof of the theorem relies on a Poisson approximation of the sum of random indicators $X_n=\sum_i X_{n,i}$ via a modification of the Chen-Stein method which is due to \cite{AGG}, and follows the bulk of the argument in \cite{crane}. Recall that the total variation distance $d_{TV}(X,Y)$ 
	between two $\nn_0$-valued random variables $X$ and $Y$ is defined as 
	\beq
	d_{TV}(X,Y)=\sup_{A\subset \nn_0}|P(X\in A)-P(Y\in A)|=\frac{1}{2}\sum_{r=0}^\infty |P(X=r)-P(Y=r)|.
	\feq
	The following summary of results in \cite{AGG} suffices for our purpose (cf. Theorem~4.2 in \cite{crane}):
	\begin{theorem}
	[\cite{AGG}]
	\label{aggth}
	Let $N\in\nn$ and $(Y_i)_{i\in [N]}$ be a collection of identically distributed (but possibly dependent) Bernoulli variables
	with $P(Y_i=1)=p\in (0,1)$ and $(Y_i=0)=1-p.$  For $i,j\in [N]$ let $p_{i,j}=E(Y_iY_j).$ Set $Y=\sum_{i=1}^N Y_i$ and $\lambda=Np.$
	For any $i\in [N]$ let $D_i\subset [N]$ be a set of indices such that
	\beq
	Y_i~\mbox{\rm is independent of}~\sigma_i,
	\feq
	where $\sigma_i$ is the $\sigma$-algebra generated by $\{Y_j:j\in D_i\},$ and define
	\beqn
	\label{b12}
	b_1=\sum_{i=1}^N p^2|D_i| \qquad \mbox{\rm and} \qquad b_2=\sum_{i=1}^N\,\sum_{j\in D_i\backslash \{i\}} p_{ij}.
	\feqn
	Let $W$ be a Poisson random variable with parameter $\lambda,$ that is $P(W=r)=\frac{\lambda^re^{-\lambda}}{r!},$ $r\in\nn_0.$  
	Then,
	\beq
	d_{TV}(Y,W)\leq 2 (b_1+b_2).
	\feq
	\end{theorem}
	We will apply Theorem~\ref{aggth} with $Y_i=X_{n,i},$ where $X_{n,i}$ are indicators introduced in the course of the proof of Theorem~\ref{wclt} 
	assuming that $k=k_n$ and $\ell=\ell_n.$ Note that under the conditions we impose, 
	\beq
	\mu_n=E(X_n)=\binom{n}{\ell_n}\binom{k_n}{d_n}\frac{1}{k_n^\ell}
	\feq
	goes to zero as $n$ tends to infinity. We have:
	\begin{theorem}
	\label{wpois}
	Suppose that three sequences of natural numbers $(k_n)_{n\in\nn},$ $(\ell_n)_{n\in\nn},$ and $(d_n)_{n\in\nn}$ satisfy the following conditions: 
	\begin{itemize}
	\item [(i)] $d_n\leq \ell_n$ and $d_n\leq k_n$ for all $n\in\nn.$ 
	\item [(ii)] $\delta:=\liminf_{n\to\infty}\frac{d_n}{\ell_n}>0.$ 
	\item [(iii)] There exist constants $A>0$ and $\beta>\frac{2}{2+\delta}$ such that $\ell_n \geq An^\beta$ for all $n\in \nn.$  
	\end{itemize}
	Consider an arbitrary sequence of patterns $v_n\in [k_n]^{\ell_n},$ $n\in\nn,$ with $d_n$ distinct letters used to form $v_n.$ 
	Let $X_n=occ_{v_n}(W_n),$ where $W_n$ is drawn at random from $[k_n]^n.$ Then 
	\beq
	\lim_{n\to\infty} d_{TV}(X_n,Q_n)=0,
	\feq
	where $Q_n$ is a Poisson random variable with parameter $\mu_n.$ In particular, 
	\beq
	\lim_{n\to\infty}\Bigl|\frac{f_r^{v_n}(k_n,n)}{k_n^n}-\frac{\mu_n^r e^{-\mu_n}}{r!}\Bigr|=0,
	\feq
	for any integer $r\geq 0.$ 
	\end{theorem}	
	\begin{remark}
	\label{rwpois}
	We believe that the lower bound for $\beta$ in the statement of the theorem is an artifact of the proof and can be improved. 
	In the most favorable to us case $\delta=1,$ the conditions of the theorem require $\beta>\frac{2}{3}.$ This is compared to the lower bound 
	$\beta>\frac{1}{2}$ obtained in \cite{crane} for permutations. 
	\end{remark}
	\begin{proof}[Proof of Theorem~\ref{wpois}]
	Fix any $n\in\nn,$ and let $K_n=\binom{n}{\ell_n}$ and $p_n=E(X_{n,j})=\frac{1}{k_n^{\ell_n}}\binom{k_n}{d_n}$ for this particular value of $n.$ Note that $\mu_n=E(X_n)=K_np_n.$ Recall the intervals $I_j$ from the proof of Theorem~\ref{wclt}, assuming that $k=k_n$ and $\ell=\ell_n,$ define for $j\in [N],$
	\beq
	Y_j=X_{n,j}\qquad \mbox{\rm and} \qquad D_j=\{m\in [K_n]:I_j\cap I_m=\emptyset\}.
	\feq
	Let $(\ell_n-i) \wedge d$ denote $\min\{\ell_n-i,d_n\}.$ Observe that if $I_j\cap I_m=i,$ then 
	\beq
	E(Y_jY_m)&=&E\bigl(Y_jE(Y_m|Y_j)\bigr)\leq E\left[Y_j\frac{1}{k_n^{\ell_n-i}}\binom{k_n}{(\ell_n-i) \wedge d_n}\right]
	\\
	&=& \frac{1}{k_n^{\ell_n} k_n^{\ell_n-i}}\binom{k_n}{d_n}\binom{k_n}{(\ell_n-i)\wedge d_n}. 
	\feq
	Therefore, for $b_1$ and $b_2$ introduced in \eqref{b12} we have: 
	\beq
	b_1=K_n\Delta_n p_n^2\leq (K_np_n)^2=\mu_n^2,
	\feq 
	where $\Delta_n$ is defined in \eqref{amount1}, and 
	\beq 
	b_2&\leq&  \frac{1}{k_n^{\ell_n}}\binom{k_n}{d_n}  \sum_{i=1}^{\ell_n-1} \binom{n}{2\ell_n-i}
	\binom{2\ell_n-i}{\ell_n}\binom{\ell_n}{i}\binom{k_n}{(\ell_n-i) \wedge d_n}\frac{1}{k_n^{\ell_n-i}}
	\\
	&=&
	 \frac{1}{k_n^{\ell_n}}\binom{k_n}{d_n} \binom{n}{\ell_n} \sum_{i=1}^{\ell_n-1} 
	\binom{n-\ell_n}{\ell_n-i}\binom{\ell_n}{i}\binom{k_n}{(\ell_n-i) \wedge d_n}\frac{1}{k_n^{\ell_n-i}}.
	\feq
	Therefore,
	\beq
	b_1+b_2 \leq
	\mu_n^2+ \mu_n\sum_{i=0}^{\ell_n}
	\binom{n-\ell_n}{\ell_n-i}\binom{\ell_n}{i}\binom{k_n}{(\ell_n-i) \wedge d_n}\frac{1}{k_n^{\ell_n-i}}.
	\feq 
	Since 
	\beq
	\binom{k_n}{(\ell_n-i) \wedge d_n}\frac{1}{k_n^{\ell_n-i}}
	\leq \frac{1}{((\ell_n-i) \wedge d_n)!}\leq \frac{1}{d_n!}+\frac{1}{(\ell_n-i)!},
	\feq
	we obtain that 
	\beq
	b_1+b_2\leq
	\mu_n^2+\frac{\mu_n}{d_n!}\binom{n}{\ell_n}+ \mu_n\sum_{m=0}^{\ell_n} \binom{n-\ell_n}{m} \binom{\ell_n}{\ell_n-m}\frac{1}{m!},
	\feq 
	where we used Vandermonde's identity for the second term and change of variables $m=\ell_n-i$ for the third one. 
	Since 
	\beq
	\mu_n=\frac{K_n}{k_n^{\ell_n}}\binom{k_n}{d_n}\leq \frac{K_n}{k_n^{d_n}}\binom{k_n}{d_n}\leq \frac{K_n}{d_n!},
	\feq 
	we obtain that 
	\beq
	b_1+b_2\leq 2\Bigl(\frac{K_n}{d_n!}\Bigr)^2+\mu_nK_n E\Bigl(\frac{1}{\Lambda_n!}\Bigr),
	\feq
	where $\Lambda_n$ is a random variable with hypergeometric distribution, $P(\Lambda_n=m)=\frac{\binom{n-\ell_n}{m} \binom{\ell_n}{\ell_n-m}}{\binom{n}{\ell_n}}$ 
	for $m=0,\ldots,\ell_n.$ By Hoeffding's inequality for partial sums of bounded random variables, 
	\beq
	P\Bigl(\Lambda_n-\ell_n\frac{n-\ell_n}{n}\leq -\veps\ell_n \Bigr) \leq e^{-2\veps^2\ell_n}
	\feq 
	for any $\veps>0.$ Thus for any given $\veps>0$ and $n$ large enough, 
	\beq
	P\bigl(\Lambda_n\leq (1-2\veps)\ell_n \bigr) \leq e^{-2\veps^2\ell_n}.
	\feq 
	Therefore, for all an arbitrary $\veps>0$ and all $n$ large enough, 
	\beq
	b_1+b_2&\leq& 2\Bigl(\frac{K_n}{d_n!}\Bigr)^2+\mu_n K_n e^{-2\veps^2\ell_n}+ \frac{\mu_n K_n}{\Gamma((1-2\veps)\ell_n)}
	\leq 2\Bigl(\frac{K_n}{d_n!}\Bigr)^2+2\mu_n K_n e^{-2\veps^2\ell_n}
	\\
	&\leq&
	2\Bigl(\frac{K_n}{d_n!}\Bigr)^2+2\frac{K_n^2 e^{-2\veps^2\ell_n}}{d_n!}\leq \frac{4K_n^2 e^{-2\veps^2\ell_n}}{d_n!},
	\feq
	where $\Gamma(\,\cdot\,)$ is the gamma function. Finally, using Stirling's formula we obtain that 
	\beq
	&& \log \Bigl(\frac{4K_n^2 e^{-2\veps^2\ell_n}}{d_n!}\Bigr)
	\\
	&&
	\quad =2\bigl\{n\log n-(n-\ell_n)\log(n-\ell_n)-\ell_n\log \ell_n -\veps^2\ell_n\bigr\}-d_n\log d_n+d_n+O(n)
	\\
	&&
	\quad =2\Bigl\{\ell_n\log \frac{n}{\ell_n}-(n-\ell_n)\log\Bigl(1-\frac{\ell_n}{n}\Bigr)\Bigr\}-d_n\log d_n+O(n)
	\\
	&&
	\quad =2\ell_n\log \frac{n}{\ell_n}-d_n\log d_n+O(n).
	\feq
	By the conditions of the theorem, $\delta=\liminf_{n\to\infty}\frac{d_n}{\ell_n}>0.$ 
	Therefore, for any $\gamma\in (0,\delta)$ and $n$ large enough we have:
	\beq
	\log(b_1+b_2)\leq \log \Bigl(\frac{4K_n^2 e^{-2\veps^2\ell_n}}{d_n!}\Bigr) \leq 2\ell_n\log n-(2+\gamma)\ell_n\log\ell_n+O(n).
	\feq
	The proof of the theorem is complete. 
	\end{proof}
		
	\section{Permutation patterns}
	\label{astrid}
	In this section, we discuss an extension of some of our results about counting occurrences of a pattern in words to permutations.
	The section is divided into two subsections. Subsection~\ref{ldpsw} is devoted to Stanley-Wilf type limits for permutations, and
	Section~\ref{wa-gf} adapts the concept of weak avoidance to permutations. The main results of this section are Theorem~\ref{thm:fnrlim} and Proposition~\ref{th1}. The latter is a counterpart of Proposition~\ref{th1words} and the former is a modification for permutations of Theorem~\ref{thm:fnrlim_word}. Extensions of the CLT-related results in Section~\ref{new} to random permutations are readily available due to the CLT for permutations proved by B\'{o}na in \cite{bona}. This is briefly discussed in the concluding paragraph of Section~\ref{wa-gf}, the details are left to the reader.
	\par
	We begin with notation. Permutations are bijections from a set $[n]$ to itself. 
	For $n\in\nn,$ let $S_n$ denote the symmetric group of order $n,$ the group of permutations of the integers in $[n]$.
	Occasionally, when confusion is not likely to occur, we will identify permutations in $S_n$ with the words representing
	the image of the permutation. For instance, for permutations $\pi=\pi(1)\cdots\pi(n)\in S_n$ and
	$\nu=\nu(1)\cdots\nu(m)\in S_m$ we refer to the permutation $$\pi\nu:=\pi(1)\cdots\pi(n)\nu(1)\cdots\nu(m)\in S_{n+m}$$
	as the concatenation of the permutations $\pi$ and $\nu.$
	\par
	Fix any $k\in \nn$ and $\xi\in S_k.$ We refer to $\xi$ as a pattern, it remains fixed throughout the rest of the paper.
	For a permutation $\pi\in S_n$ with $n\geq k,$ an occurrence of the pattern $\xi$ in $\pi$ is a sequence of $k$ indices
	$1\leq i_1<i_2<\dots<i_k\leq n $ such that the word $\pi(i_1)\cdots \pi(i_k)\in [n]^k$ is order-isomorphic to the word $\xi,$
	that is
	\beq
	\pi(i_p)<\pi(i_q)\Longleftrightarrow  \xi_p<\xi_q\qquad \forall\,1\leq p,q\leq k.
	\feq
	For a permutation $\pi\in S_n$ with $n\geq k$ we denote by $occ_\xi(\pi)$ the number of occurrences of the pattern $\xi$ in $\pi.$ For example,
	if $\xi=12$ and $\pi=51324,$ then $13,$ $12,$ $14,$ $34,$ and $24$ are order-isomorphic to $12,$ and $occ_\xi(\pi)=5.$ If $occ_\xi(\pi)=m,$ we
	say that $\pi$ contains $\xi$ (exactly) $m$ times. For a given $r\in \nn_0,$ let $f_r^\xi(n)$ denote the number of permutations in $S_n$ that contain $\xi$ exactly $r$ times. That is,
	\beq
	f^\xi_r(n) = \#\{ \pi\in S_n : occ_\xi(\pi) = r \} ,\qquad r\geq 0.
	\feq
	For example, if $\xi=12$ then $f_0^\xi(3)=1$ (only $321$ counts), $f_1^\xi(3)=2$ ($312$ and $231$ count),
	$f_2^\xi(3)=2$ ($132$ and $213$ count), and $f_3^\xi(3)=1$ (only $123$ counts).
	\par 
	As in Section~\ref{astrid}, $a_n\sim b_n,$ $a_n=O(b_n)$ and $a_n=o(b_n)$ for sequences $a_n$ and $b_n$ with elements that might depend on $k,r,\xi$ and 
	other parameters, means that, respectively, $\lim_{n\to\infty}\frac{a_n}{b_n}=1,$ $\limsup_{n\to\infty}\bigl|\frac{a_n}{b_n}\bigr|<\infty,$ 
	and $\lim_{n\to\infty}\frac{a_n}{b_n}=0$ for all feasible values of the parameters when the latter are fixed. The notation 
	$a_n=\Theta(b_n)$ is used to indicate that both $a_n=O(b_n)$ and $b_n=O(a_n)$ hold true.   
	\subsection{Stanley-Wilf type limits}
	\label{ldpsw}
	The celebrated Stanley-Wilf conjecture proved in \cite{MT} states that $\lim_{n\to\infty}\frac{1}{n}\log f_0^\xi(n)$ exists
	and belongs to $(0,\infty).$ For $\pi\in S_n,$ let $Z_n=occ_\xi(\pi),$
	where $\pi$ is a permutation chosen at random uniformly over $S_n.$ Notice that
	\beq
	P(Z_n=r)=\frac{f_r^\xi(n)}{n!},\qquad r\in\nn_0.
	\feq 	
	In the language of random permutations, the Stanley-Wilf limit is   	
	\beq
	\lim_{n\to\infty}\frac{1}{n}\log [n!P(Z_n=0)]=\lim_{n\to\infty}\Bigl(\frac{1}{n}\log P(Z_n=0)+\log n -1\Bigr),
	\feq
	which yields the following weaker conclusion:
	\beq
	\lim_{n\to\infty}\frac{1}{n\log n}\log P(Z_n=0)=-1.
	\feq
	Thus the limit can be interpreted in terms of the asymptotic behavior of $P(Z_n=0)$ as a local large deviation result with respect
	to the scaling sequence $n\log n.$ The probability $P(Z_n=0)$ is very small since according to the CLT obtain by B\'{o}na in \cite{bona},
	$Z_n$ is tightly concentrated around $E(Z_n)=\frac{1}{k!}\binom{n}{k}.$ The following theorem extends this large deviation result to $P(Z_n=r)$ with an arbitrary fixed $r\in\nn.$
	\begin{theorem}
	\label{thm:fnrlim}
	For any $r\in\nn,$ $\lim_{n\to\infty}(f_r^\xi(n))^{\frac{1}{n}}$ exists and is equal to $\lim_{n\to\infty} (f_0^\xi(n))^{\frac{1}{n}}.$
	\end{theorem}
	\begin{proof}
	The proof by induction on $r$. By Corollary~2 in \cite{MT}, $c:=\lim_{n\to \infty} \bigl(f_0^\xi(n)\bigr)^{\frac{1}{n}}$ exists and is finite.
	Assume that for some $m\in\nn$ the claim holds for $r=0,1,\ldots,m-1.$ To complete the proof, we need to show that under this assumption it holds also
	for $r=m.$
	\par
	To this end, let $\pi$ be an arbitrary permutation in $S_n$ that contains the pattern $\xi$ exactly $m$ times. By removing the leftmost letter in the leftmost occurrence of $\xi$ in $\pi$ and renaming the remaining letters, we obtain a permutation $\pi'$ in $S_{n-1}$ that contains $\xi$
	at most $m-1$ times. Thus,
	\beq
	f_m^\xi(n)\leq n\sum_{j=0}^{m-1}f_j^\xi(n-1).
	\feq
	It follows that
	\begin{align}
	\limsup_{n\to\infty}(f_m^\xi(n))^{1/n}\leq \lim_{n\to\infty}\Bigl(n\sum_{j=0}^{m-1}
	f_j^\xi(n-1)\Bigr)^{1/(n-1)}=c.
	\label{eqrr1}
	\end{align}
	On the other hand, consider an arbitrary permutation $\pi\in S_n$ that contains $\xi$ exactly $m-1$ times
	and the concatenation $\pi' = \pi \xi' \in S_{n+k},$ where $\xi'$ is obtained by adding $n$ to each letter in $\xi.$ For instance, if $n=5,$
	$\pi=13542,$ and $\xi=12,$ then $\xi'=67$ and $\pi'=1354267.$ Without loss of generality, we may assume that the letter
	$k$ precedes $1$ in $\xi$ (the idea is borrowed from \cite{Arr}). Because of this assumption, the new permutation
	$\pi'$ contains $\xi$ exactly $m$ times. We can therefore conclude that $f_{m-1}^\xi(n)\leq f_m^\xi(n+k).$
	This inequality along with the induction hypothesis imply that
	\beq
	c = \lim_{n\to\infty}(f_{m-1}^\xi(n))^{1/n}\leq
	\liminf_{n\to\infty}\bigl(f_m^\xi(n+k)\bigr)^{1/n} =  \liminf_{n\to\infty}\bigl(f_m^\xi(n)\bigr)^{1/n}.
	\feq
	In view of \eqref{eqrr1}, this completes the proof of the theorem.
	\end{proof}
	\subsection{Weak avoidance of permutation patterns}
	\label{wa-gf}	
	Similarly to \eqref{cw}, with any pattern $\xi\in S^k$ one can associate a sequence of weak avoidance penalty functions
	$c^\xi_n:[0,1]\to [0,n!],$ $n\in\nn,$ by setting
	\beqn
	\label{cw1}
	c^v_{k,n}(x) = \sum_{\pi \in S_n}\, \prod_{1\leq j_1 < \cdots < j_k \leq n} \left( 1 + x V_{j_1,\cdots, j_k} (\xi,\pi) \right),
	\feqn
	where
	\beq
	V_{j_1, \cdots, j_k}(\xi, \pi) =
	\left\{
	\begin{array}{cl}
	-1 & \quad \mbox{\rm if} \quad  \bigl(\,\pi(j_q) < \pi(j_r)  ~\Longleftrightarrow ~ \xi(q)<\xi(r)\quad \forall\,1\leq q,r\leq k\,\bigr)\\
	0 & \quad \mbox{\rm otherwise.}
	\end{array}
	\right.
	\feq
	Notice that $c^\xi_n(0) = n!$ and $c^\xi_n(1) = f_0^\xi(n)$.
	Similarly to \eqref{th2words}, we have
	\beqn
	\label{th3words}
	c^v_n(x) =\sum_{\pi \in S_n} (1-x)^{occ_\xi(\pi)}=\sum_{r\geq 0} f_r^\xi(n)(1-x)^r.
	\feqn
	For certain particular cases the polynomials $c^v_n(1-x),$ generating functions of the sequence $f_r^\xi(n),$ $n\in\nn,$ 
	have been studied in \cite{janson1,nakamura} through the analysis of certain recursive functional equations that they satisfy.
	\par
	The analogue of the $\qq^{v,x}_{k,n}$ measure introduced in \eqref{qword} is the probability measure $\pp^{v,x}_n$ on $S_n$ defined by
	\beq
	\pp^{\xi,x}_n(A)=\frac{1}{c^v_n(x)}\sum_{\pi \in A}(1-x)^{occ_\xi(\pi)},\qquad A\subset S_n.
	\feq
	In the case of inversions, i.\,e. for $\xi=21,$ $\pp^{\xi,x}_n$ is a Mallow's distribution. Mallow's permutations have been
	studied by several authors, see, for instance, recent \cite{crane, mall, pitman} and references therein.
	\par
	The next proposition establishes the existence of $\lim_{n\to\infty}\bigl(c_n^{x}(\xi)\bigr)^{1/n}.$
	The proof is based on a standard sub-additivity argument, and follows the same line of argument as the one in \cite{Arr}.
	Unfortunately, we were unable to verify that the limit is necessarily finite (cf. Proposition~\ref{th1words} together with \eqref{wlima} for words). 	
	\begin{proposition}
	\label{th1}
	$\lim_{n\to\infty}\bigl(c^\xi_n(x)\bigr)^{\frac{1}{n}}$ exists for all $x\in[0,1].$
	\end{proposition}
	\begin{proof}
	For $\pi\in S_n$ and $i,j\in\nn$ such that $1\leq i<j\leq n,$ let
	\beq
	\pi_{i,j}=\pi_i(i)\cdots \pi_i(j),\qquad \mbox{\rm where}~\pi_i(r):=\pi(r)-i+1.
	\feq
	That is $\pi_{i,j}\in [n]^{j-i+1}$ and $\pi_{i,j}(r)=\pi(i-1-r)-(i-1)$ for all $r\in[n].$
	Further, for any $m,n\in\nn$ such that $m\leq n$ let
	\beq
	S_n^m=\{\pi\in S_n:\pi_{1,m}\in S_m\}.
	\feq
	Note that $\pi\in S_n^m$ implies $\pi_{m+1,n}\in S_{n-m}.$ In other words,
	\beqn
	\label{bij}
	\pi\to (\pi_{1,j},\pi_{j+1,n})~\mbox{\rm is a bijection between}~S^j_n~\mbox{\rm and}~S_j\times S_{n-j}.
	\feqn
	Without loss of generality, we can and will assume that $\xi^{-1}(k)<\xi^{-1}(1),$ that is $k$ appears before $1$ in $\xi.$
	Under this assumption, we have
	\beqn
	\label{impa}
	\bigl\{\pi\in S_n^m, j_1\leq m,j_k>m\}\Longrightarrow U_{j_1,\cdots, j_k} (\xi,\pi)=0.
	\feqn
	In view of \eqref{bij} and \eqref{impa}, for any $n,m\in\nn$ and $x\in [0,1]$ we have
	\beq
	&&
	c_{n+m}^\xi(x) = \sum_{\pi \in S_{n+m}}\, \prod_{1\leq j_1 < \cdots < j_k \leq n+m} \bigl( 1 + x U_{j_1,\cdots, j_k} (\xi,\pi)\bigr)
	\\
	&&
	\quad
	\geq  \sum_{\pi \in S_{n+m}^m}\, \prod_{1\leq j_1 < \cdots < j_k \leq n+m} \bigl( 1 + x U_{j_1,\cdots, j_k} (\xi,\pi) \bigr)
	\\
	&&
	\quad = \sum_{\pi_1 \in S_m}\, \prod_{1\leq j_1 < \cdots < j_k \leq m}\bigl( 1 + x U_{j_1,\cdots, j_k} (\xi,\pi_1) \bigr)
	\sum_{\pi_2 \in S_n}\, \prod_{1\leq j_1 < \cdots < j_k \leq n} \bigl( 1 + x U_{j_1,\cdots, j_k} (\xi,\pi_2) \bigr)
	\\
	&&
	\quad
	=c^\xi_m(x)c^\xi_n(x).
	\feq
	Hence, $-\log c^\xi_n(x),$ $n\in\nn,$ is a subadditive sequence, and by Fekete's subadditive lemma,
	$\lim_{n \to \infty} \bigl(c^\xi_n(x)\bigr)^{\frac{1}{n}}$ exists for all $x \in [0,1].$
	\end{proof}
	\begin{example}
	\label{mahonian}	
	Consider $\xi=21.$ Then the number of occurrences of $\xi$ in a permutation $\pi$ is the number of inversions in $\pi,$
	and $f_r^{21}(n)$ are Mahonian numbers \cite{Bbook}. The identity in \eqref{th3words} together with Netto's formula for the generating function of the
	sequence $\{f_r^{21}(n):r\geq 0\}$ (see, for instance, \cite[p.~43]{Bbook} or \cite[Seq A008302]{Slo}) give
	$c_n^{21}(x)=\prod_{j=1}^n\frac{1-(1-x)^j}{x}.$ In particular, $\lim_{n\to\infty}\bigl(c_n^{21}(x)\bigr)^{1/n}=x^{-1}$ for all $x\neq0$. Note that
	$f_0^{21}(n)=1$  for all $n\in\nn,$ and hence by virtue of Theorem~\ref{thm:fnrlim}, $\lim_{n\to\infty} \bigl(f_r^{21}(n)\bigl)^{1/n}=1$ for all $r\in\nn.$
	Interestingly enough, in contrast to Example~\ref{inv}, the asymptotic behavior of $c_n^{21}(x_n)$ for a sequence $x_n$ such that $x_n\sim 1$ as $n\to\infty,$ does depend on the rate of convergence of $x_n.$
	\end{example}
	\begin{conj*}
	\label{c3}
	$\lim_{n \to \infty} \bigl(c^\xi_n(x)\bigr)^{\frac{1}{n}}<\infty$ for all patterns $\xi\in \cup_k S_k$ and all $x\in (0,1).$
	\end{conj*}
	It is interesting to notice that while for words we have $c_r^v(k,n+m)\leq c_r^v(k,n)c_r^v(k,m),$ the opposite is true for permutations, namely $c_r^\xi(n+m)\geq  c_r^\xi(n)c_r^\xi(m).$ The differences can be explained as follows. For words we have: 
	\beq 
	c_r^v(k,n+m)=E\Bigl[E\bigl[(1-x)^{X_{n+m}}\bigr]\,\Bigl|\,W_n\Bigr]=E\Bigl[(1-x)^{X_n}E\bigl[(1-x)^{X_{n+m}-X_n}\bigr]\,\Bigl|\,W_n\Bigr],
	\feq
	and, since letters can be repeated in words, the conditional expectation is less than the unconditional one $E\bigl[(1-x)^{X_m}\bigr].$ Indeed, any pattern occurrence in the first $n$ letters does not affect the last $m$ letters in $W_{n+m},$ but does increase the probability of having occurrences of the pattern spread over two intervals, $[1,n]$ and $[n+1,n=m].$ It turns out that with permutations, where letters cannot be re-used, the situation is different and the correlation between occurrences of the pattern in the beginning and continuation of a large permutation is negative in contrast to words. \par  
	We conclude with a remark concerning the extension of the results in Section~\ref{new} to permutations.
	The key elements in the proofs in Section~\ref{new} is the specific covariance structure (the dependence graph) of the indicators $X_{n,i}$ and the asymptotic relation $\frac{\mu_n}{\sigma_n}=\Theta(\sqrt{n})$ between the expectation and variance of $X_n.$ B\'{o}na's CLT for permutations \cite{bona} asserts that the key elements are similar for words and permutations, and thus  enables one to carry over the proofs of Corollaries~\ref{bess}, \ref{wsldp}, and \ref{wldp}, Lemma~\ref{cvclt}, and Theorem~\ref{wpois} to permutations nearly verbatim. We leave the details to the reader.
	
	\end{document}